\documentclass[letterpaper, 10 pt, conference]{ieeeconf}
\IEEEoverridecommandlockouts
\usepackage[utf8]{inputenc}
\usepackage{textcomp}

\usepackage{amsmath,bm}
\usepackage{amsfonts}
\usepackage{amssymb}
\usepackage{graphicx}
\usepackage{xcolor}
\usepackage{cite}
\usepackage{subfig}
\usepackage{hyperref}
\hypersetup{
  colorlinks  = true, 
  urlcolor    = blue,
  linkcolor   = blue, 
  citecolor   = blue 
}
\newtheorem{Definition}{Definition}
\newtheorem{Theorem}{Theorem}

\newtheorem{Lemma}{Lemma}

\newcommand{\m}[1]{\mathbf{#1}}
\newcommand{\mc}[1]{\mathcal{#1}}
\newcommand{\mb}[1]{\mathbb{#1}}

\usepackage[switch,pagewise]{lineno}
\begin{document}

\title{\LARGE \bf Matrix-Scaled Consensus}
\author{Minh Hoang Trinh\authorrefmark{1}, Dung Van Vu\authorrefmark{2}, Quoc Van Tran\authorrefmark{3}, Hyo-Sung Ahn\authorrefmark{4},
\thanks{${}^{*}$Department of Automation Engineering, School of Electrical and Electronic Engineering, Hanoi University of Science and Technology (HUST), Hanoi 11615, Vietnam. E-mail: \tt\footnotesize{minh.trinhhoang@hust.edu.vn}}
\thanks{${}^{\dagger}$Unmanned Aerial Vehicle Center, Viettel High Technology Industries Corporation, Hanoi 11209, Vietnam. E-mail:
\tt\footnotesize{vuvandung.bkhn@gmail.com}}
\thanks{${}^{\ddagger}$Department of Mechatronics, School of Mechanical Engineering, Hanoi University of Science and Technology (HUST), Hanoi 11615, Vietnam. E-mail: \tt\footnotesize{tvquoc9790@gmail.com,~quoc.tranvan@hust.edu.vn}}
\thanks{${}^{\S}$School of Mechanical Engineering, Gwangju Institute of Science and Technology (GIST), Gwangju 61005, Republic of Korea. E-mail: \tt\footnotesize{hyosung@gist.ac.kr}}
}

\maketitle
\thispagestyle{empty}
\begin{abstract}
This paper proposes \emph{matrix-scaled consensus} algorithm, which generalizes the scaled consensus algorithm in \cite{Roy2015scaled}. In (scalar) scaled consensus algorithms, the agents' states do not converge to a common value, but to different points along a straight line in the state space, which depends on the scaling factors and the initial states of the agents. In the matrix-scaled consensus algorithm, a positive/negative definite matrix weight is assigned to each agent. Each agent updates its state based on the product of the sum of relative matrix scaled states and the sign of the matrix weight. Under the proposed algorithm, each agent asymptotically converges to a final point differing with a common consensus point by the inverse of its own scaling matrix. Thus, the final states of the agents are not restricted to a straight line but are extended to an open subspace of the state-space. Convergence analysis of matrix-scaled consensus for single and double-integrator agents are studied in detail. Simulation results are given to support the analysis.
\end{abstract}
%

\section{Introduction}
Consensus algorithm and its variations \cite{Ren2007magazine,Freeman2006stability,Sarlette2009,Altafini2013,Roy2015scaled,Trinh2018matrix} have been the main model for studying networked systems. Though simple, consensus algorithms can describe intricate phenomena such as bird flocking, synchronization behaviors, or how a group of people eventually reaches an agreement after discussions \cite{Jadbabaie2003coordination,Li2009consensus,Proskurnikov2017tutorial,Ye2020Aut}. The consensus algorithm is also used to coordinate large-scale systems such as formation of vehicles, electrical, sensor, and traffic networks \cite{Olfati2007consensuspieee}.

Consider a network in which the interactions between subsystems, or agents, is modeled by a graph. In the consensus algorithm, each agent updates its state based on the sum of the relative states with its nearby agents. If the interaction graph is connected, the agents' states asymptotically converge to a common point in the space, and we say that the system asymptotically reaches a consensus. 

The author in \cite{Roy2015scaled} proposed a scaled consensus model, in which each agent has a scaling gain $s_i$ and updates its state variable $x_i$ based on the consensus law\footnote{Notations will be defined in detail in Section~\ref{sec:2}.}
\begin{align} \label{eq:ScaledConsensus}
\dot{x}_i = \text{sign}(s_i) \sum_{j \in \mc{N}_i } (s_j x_j - s_i x_i), \quad ~i=1,\ldots, n.
\end{align}
The system \eqref{eq:ScaledConsensus} achieves a scaled-consensus globally asymptotically, that is, $s_i x_i(t) \to s_j x_j(t)$, as $t\to \infty$ and agents with the same $s_i$ will converge to the same point (or cluster). The system  \eqref{eq:ScaledConsensus} can describe a cooperative network, where agents have different levels of consensus on a single topic. Further extensions of the scaled consensus algorithm  with consideration to switching graphs, time delays, disturbance attenuation, or different agents' models can be found in the literature, for examples, see  \cite{Meng2015scaled,Meng2015TIE,Aghbolagh2016scaled,Shang2017delayed,Hanada2019new,Wu2021adaptive}.

This paper generalizes the consensus model \eqref{eq:ScaledConsensus} by assuming that each agent has a state vector and a positive or negative definite scaling matrix. The proposed model has some interesting features. First, thanks to the matrix weights, the system still achieves clustering behavior, but the final states are not restricted to be distributed along a straight line. Under the matrix-scaled consensus algorithm, a virtual consensus point is jointly determined by the initial states and the scaling matrices of all agents. The state vector of each agent converges to a point differently from the virtual consensus point by the inverse of its scaling matrix. As a result, clustering behaviors usually happen, and agents with the same scaling matrix converge to a common cluster in the space. Second, although the proposed consensus law has similarities with the biased consensus \cite{Ahn2019consensus} and orientation estimation algorithms \cite{lee2016distributed,Tran2018ecc}, in the proposed model, the scaling matrices are not limited to the rotation matrices. Under the assumption that the interaction graph is undirected and connected, it is shown that if the scaling matrices are positive definite or negative definite (possibly asymmetric), then the system achieves a matrix-scaled consensus globally asymptotically. Finally, the proposed algorithm can be used as a  multi-dimensional model for studying clustering behaviors in a social network. Unlike the matrix-weighted consensus algorithm \cite{Trinh2018matrix}, in which a positive definite/semidefinite matrix weight associated with an edge in the graph characterizes the degree of cooperation between the agents in the network, in the matrix-scaled consensus algorithm, a positive/negative definite matrix weight corresponds to a vertex in the graph and represents a local coordinate system of each agent. Each local coordinate system can be interpreted as the private belief system of an individual on $d$ logically dependent topics, and this belief system is usually not perfectly aligned with a social norm (a global coordinate system). If each individual updates his/her opinion based on his/her own belief system, due to the existence of negative definite scaling matrices, the system will be unstable. By self-realizing the negative/positive of his/her own belief's system relative to a social norm, each individual adjusts the opinion along a direction that is not contrary to the social norm. The sign of the scaling matrix in the proposed algorithm, thus, realizes the readiness of each individual to compromise in order to prevent the society from divergence. Moreover, we propose a matrix-scaled consensus algorithm  for a network of double integrator agents. A corresponding convergence condition related to the damping gain and the eigenvalues of the scaled Laplacian is also given. 

The remainder of this paper is organized as follows. Section~\ref{sec:2} provides notations and the theoretical framework that will be used throughout the paper. The matrix-scaled consensus algorithms for single and double integrator agents are proposed and examined in Sections~\ref{sec:3} and \ref{sec:4}, respectively. Simulation results are given in Section~\ref{sec:5} and Section~\ref{sec:6} concludes the paper. 
\section{Preliminaries}
\label{sec:2}
\subsection{Notations}
The sets of real and complex numbers are denoted by $\mb{R}$ and $\mb{C}$, respectively. Scalars are denoted by lowercase letters, while bold font normal and capital letters are used for vectors and matrices, respectively. 
Let $\m{A} \in \mb{R}^{m\times n}$, its transpose is given by $\m{A}^\top$. The kernel, image, rank, and determinant of $\m{A}$ are denoted by ker$(\m{A})$, im$(\m{A})$, rank($\m{A}$), and $\text{det}(\m{A})$, respectively. For a vector $\m{x} = [x_1, \ldots, x_d]^\top$, its 2-norm is $\|\m{x}\| = \sqrt{\sum_{i=1}^{d} x_i^2}$. Let $\m{x}_1, \ldots, \m{x}_n \in \mb{R}^d$, the vectorization operator is defined as vec$(\m{x}_1,\ldots,\m{x}_n) = [\m{x}_1^\top,\ldots,\m{x}_n^\top]^\top \in \mb{R}^{dn}$. A matrix $\m{A} \in \mb{R}^{d\times d}$ is positive definite (negative definite) if and only if $\forall \m{x} \in \mb{R}^d$, $\m{x}\neq \m{0}_d$, then $\m{x}^\top \m{A} \m{x}>0$ (resp., $\m{x}^\top \m{A} \m{x} < 0$). 
\subsection{Useful lemma}
\begin{Lemma} \cite{Li2009consensus} \label{lem:complex_quadratic_equation} The complex-coefficient polynomial $p(s) = s^2 + (a + bj) s + c + jd$, where $j^2 = -1$, is Hurwitz if and only if $abd + a^2c-d^2>0$.
\end{Lemma}
\subsection{Algebraic graph theory}
An undirected graph is given by $\mc{G}=(\mc{V},\mc{E})$, where $\mc{V}=\{1, \ldots, n\}$ is the vertex set and $\mc{E} \subset \mc{V}^2$ is the set of $|\mc{E}| = m$ edges. If there is an edge $(i,j) \in\mc{E}$ connecting vertices $i, j \in \mc{V}$, $i\neq j$, then $i$ and $j$ are adjacent to each other. The neighbor set of a vertex $i$, denoted by $\mc{N}_i = \{j \in V|~(i,j) \in \mc{E} \}$, contains all adjacent vertices of $i$. A path in $\mc{G}$ is a sequence of edges connecting adjacent vertices in the graph. A graph is \emph{connected} if and only if there is a path between any pair of vertices in $\mc{V}$. Let the edges be indexed as $\mc{E} = \{e_1,\ldots, e_m\}$, and oriented such that for each edge $e_k = (i,j)$, $i$ is the starting vertex and $j$ is the end vertex of $e_k$. The \emph{incidence matrix} $\m{H}=[h_{kl}] \in \mb{R}^{m\times n}$ has $h_{kl} = -1$ if $l = i$, $h_{kl} = 1$ if $l=j$, and $h_{kl} = 0$, otherwise. The \emph{Laplacian matrix} $\m{L} = [l_{ij}] \in \mb{R}^{n \times n}$ of $\mc{G}$ is defined as follows:
\[{l_{ij}} = \left\{ {\begin{array}{*{20}{cl}}
{ - 1,}&{i \ne j,~(i,j)\in \mc{E},}\\
{~0,}&{i \ne j,~(i,j)\notin \mc{E},}\\
{ - \sum\nolimits_{i = 1,i \ne j}^n {{l_{ij}},} }&{i = j.}
\end{array}} \right.\]
As $\mc{G}$ is undirected, $\m{L}$ is symmetric positive semidefinite with eigenvalues given by $0 = \lambda_1 \le \lambda_2 \le \ldots \le \lambda_n$. We can write $\m{L} = \m{H}^\top\m{H}$. For a connected graph, $\lambda_2 > 0$ and $\text{ker}(\m{L})=\text{ker}(\m{H})=\text{im}(\m{1}_n)$.

\subsection{Matrix-scaled consensus}
Consider a multi-agent system consisting of $n$ agents. Each agent $i\in \{1, \ldots, n\}$ has a state vector $\m{x}_i \in \mb{R}^{d} $ $(d \geq 2)$ and a scaling matrix $\m{S}_i \in \mb{R}^{d\times d}$, which is either positive definite or negative definite.\footnote{Note that $\m{S}_i$ is not required to be symmetric.} Define the signum function 
\[\text{sign}(\m{S}_i) = \left\{ {\begin{array}{*{20}{c}}
~1, & \m{S}_i \text{ is positive definite,}\\
{ - 1,}& \m{S}_i \text{ is negative definite.}
\end{array}} \right.\]
Then $|\m{S}_i| \triangleq \text{sign}(\m{S}_i) \m{S}_i$ is a positve definite matrix. It is worth noting that $\text{sign}(\m{S}_i) = \text{sign}(\m{S}_i^\top) = \text{sign}(\m{S}_i^{-1})$ and $|\m{S}_i^{-1}|=|\m{S}_i|^{-1}$. 
Let $\m{x} = \text{vec}(\m{x}_1,\ldots, \m{x}_n)$, the following definition will be used in this paper:
\begin{Definition} The $n$-agent system achieves a \emph{matrix scaled consensus (MSC)} in a state $\m{x}$ if and only if $\m{x}\in \mc{A}$, where
\begin{align} \label{eq:MSC_definition}
\mc{A} = \{\m{x}\in \mb{R}^{dn}|~\m{S}_1 \m{x}_1 = \m{S}_2\m{x}_2 = \ldots = \m{S}_n \m{x}_n = \m{x}^a\},
\end{align}
and $\m{x}^a \in \mb{R}^d$ is called the virtual consensus point of the system.
\end{Definition}

Equivalently, the $n$-agent system achieves a matrix-scaled consensus  if and only if 
\begin{align}
\m{x}_i = \m{S}_i^{-1} \m{x}^a = \m{S}_i^{-1} \m{S}_j \m{x}_j,~\forall i,j \in \mc{V}.
\end{align}

\section{Matrix-scaled consensus of single-integrator agents}
\label{sec:3}
\subsection{The proposed consensus law}
Consider a system of $n$ single-integrator modeled agents in $\mb{R}^{d} $ $(d \geq 2)$: 
\[\dot{\m{x}}_i = \m{u}_i,~i = 1, \ldots, n,\]
where $\m{x}_i, \m{u}_i \in \mb{R}^d$ are respectively the state variable and the input of agent $i$. The matrix-scaled consensus is proposed as follows:
\begin{align} 
\m{u}_i &= \text{sign}(\m{S}_i) \sum_{j\in \mc{N}_i} (\m{S}_j \m{x}_j - \m{S}_i \m{x}_i), \nonumber\\
&= |\m{S}_i| \sum_{j\in \mc{N}_i} (\m{S}_i^{-1}\m{S}_j \m{x}_j - \m{x}_i),~i = 1, \ldots, n. \label{eq:MSC_agent_i}
\end{align}
In the consensus algorithm \eqref{eq:MSC_agent_i}, each agent $i$  measures the relative state vector $\m{S}_i^{-1}\m{S}_j\m{x}_j$ from its neighboring agents, sums up the relative vector $\m{r}_i = \sum_{j\in \mc{N}_i} (\m{S}_i^{-1}\m{S}_j \m{x}_j - \m{x}_i)$, and updates its state variable along the direction of $|\m{S}_i|\m{r}_i$.  The $n$-agent system can be written in the compart form as follows:
\begin{align} \label{eq:MSC_x}
\dot{\m{x}} = -\big(\text{diag}(\text{sign}(\m{S}_i))\m{L}\otimes \m{I}_d \big) \text{blkdiag}(\m{S}_i)\m{x},
\end{align}
where $\m{x}=\text{vec}(\m{x}_1,\ldots,\m{x}_n)$, $\otimes$ stands for the Kronecker product and $\m{S} \triangleq \text{blkdiag}(\m{S}_i)$ is the block diagonal matrix with the matrices $\m{S}_1,\ldots, \m{S}_n$ in the main diagonal. Introduce the variable  transformation $\m{x}_c = \m{S} \m{x}$ and let $\m{\Theta} = |\m{S}|\bar{\m{L}}$, where $\bar{\m{L}} = \m{L}\otimes \m{I}_d$ and $|\m{S}| =  \text{blkdiag}(|\m{S}_i|)$. We can reexpress the $\m{x}_c$-dynamics in the matrix form as follows:
\begin{align}  \label{eq:MSC_y}
\dot{\m{x}}_c = -\m{\Theta} \m{x}_c.
\end{align}
Note that if the signum term is omitted, the system~\eqref{eq:MSC_y} becomes $\dot{\m{x}}_c = -\m{S}\bar{\m{L}}\m{x}_c$. Since $\m{S}\bar{\m{L}}$ is the product of a nondefinite matrix $\m{S}$ with a positive semidefinite matrix $\bar{\m{L}}$, it will be likely that $\m{S}\bar{\m{L}}$ contains eigenvalues with positive real parts, and the system is unstable.
\subsection{Stability analysis}
We will study the system \eqref{eq:MSC_y} in this subsection. Since $|\m{S}_i|$, $i\in \mc{V}$, are positive definite, $\m{|S|}=\text{blkdiag}(|\m{S}_i|)$ is positive definite.  Thus, $\text{rank}(\m{\Theta}) = \text{rank}(\bar{\m{L}}) = dn-d$ and $\text{ker}(\m{\Theta})=\text{ker}(\bar{\m{L}})=\text{im}(\m{1}_n\otimes \m{I}_d)$. The following lemma characterizes the spectrum of $\m{\Theta}$:
\begin{Lemma} \label{lem:spectral_M} Suppose that $\mc{G}$ is undirected and connected. The matrix $\m{\Theta}$ has $d$ zero eigenvalues and $dn-d$ eigenvalues with positive real parts.
\end{Lemma}
\begin{proof}
Following the proof of \cite{Godsil2001}[Lem. 8.2.4], we define the following matrices $\bar{\m{H}} = \m{H}\otimes \m{I}_d$,
\begin{align*}
\m{X} = \begin{bmatrix}
\m{I}_{dn} & s^{-1}\m{|S|} \bar{\m{H}}^\top\\
\bar{\m{H}} & \m{I}_{dm}
\end{bmatrix},~\text{and }
\m{Y} = \begin{bmatrix}
\m{I}_{dn} &  \m{0}_{dn\times dm}\\
-\bar{\m{H}} & \m{I}_{dm}
\end{bmatrix}.
\end{align*}
Then,
\begin{align*}
\m{X}\m{Y} &= \begin{bmatrix}
\m{I}_{dn}-s^{-1}\m{|S|} \bar{\m{L}} & s^{-1}\m{|S|}\bar{\m{H}}^\top\\
\m{0}_{dm\times dn} & \m{I}_{dm}
\end{bmatrix}, \\ 
\m{Y}\m{X} &= \begin{bmatrix}
\m{I}_{dn} & s^{-1} \m{|S|}\bar{\m{H}}^\top\\
\m{0}_{dm\times dn} & \m{I}_{dm}-s^{-1}\bar{\m{H}}\m{|S|} \bar{\m{H}}^\top
\end{bmatrix}.
\end{align*}
From the fact that $\text{det}(\m{X}\m{Y}) = \text{det}(\m{Y}\m{X})$, one has
\begin{align} \label{eq:Sylvester}
\text{det}(\m{I}_{dn}-s^{-1}\m{|S|} \bar{\m{L}}) &= \text{det}(\m{I}_{dm}-s^{-1}\bar{\m{H}}\m{|S|} \bar{\m{H}}^\top) \nonumber\\
s^{d(m-n)}\text{det}(s\m{I}_{dn}-\m{|S|} \bar{\m{L}}) &= \text{det}(s\m{I}_{dm}-\bar{\m{H}}\m{|S|} \bar{\m{H}}^\top)
\end{align}
Thus, the nonzero eigenvalues of two matrices $\m{\Theta} = \m{|S|} \bar{\m{L}}$ and $\m{N} = \bar{\m{H}}\m{|S|} \bar{\m{H}}^\top$ are the same. Since $(\m{|S|}$ is positive definite, $\m{N}+\m{N}^\top = \bar{\m{H}}(\m{|S|} + \m{|S|}^\top)\bar{\m{H}}^\top$ is symmetric and positive semidefinite, $\text{rank}(\m{N}+\m{N}^\top) = \text{rank}(\bar{\m{H}})=dn-d$. Therefore, $\m{\Theta}$ has $dn-d$ eigenvalues with positive real parts.
\end{proof}

Let $\m{\Theta}$ be expressed in the Jordan canonical form $\m{\Theta} = \m{W} \m{J} \m{W}^{-1}$, where $\m{W} = [\m{w}_1, \ldots, \m{w}_{dn}] \in \mb{C}^{dn\times dn}$. For brevity, the notation $\m{W}_{[j:k]} = [\m{w}_j,\ldots,\m{w}_k]$ is adopted to denote the columns from $j$ to $k$ ($j < k$) of the matrix $\m{W}$. 

By selecting $\m{W}_{[1:d]} = \m{1}_n \otimes \m{I}_d$,~we have $((\m{W}^{-1})^\top)_{[1:d]}=(|\m{S}^{-1}|)^\top (\m{1}_n \otimes \m{I}_d)\m{P}^\top$, 
where $|\m{S}^{-1}| = \text{blkdiag}(|\m{S}_i^{-1}|)$ and $\m{P} \in \mb{R}^{d\times d}$ is included so that the normalization condition
\begin{align} \label{eq:normalization}
(((\m{W}^{-1})^\top)_{[1:d]})^\top \m{W}_{[1:d]} = \m{I}_d
\end{align}
is satisfied. The equation \eqref{eq:normalization} is equivalent to
\begin{align}
\m{P}(\m{1}_n^\top \otimes \m{I}_d)|\m{S}^{-1}| (\m{1}_n \otimes \m{I}_d) = \m{I}_d,
\end{align}
and it follows that $\m{P} = \Big(\sum_{i=1}^n |\m{S}_i^{-1}|\Big)^{-1}$. Since all nonzero eigenvalues of $-\m{\Theta}$ have negative real parts, there holds
\begin{align}
\lim_{t\to +\infty} \m{x}_c(t) & = \lim_{t\to+\infty} \text{exp}\left({-\m{\Theta}t}\right) \m{x}_c(0) \nonumber\\
&= (\m{1}_n \otimes \m{I}_d) \m{P} (\m{1}_n^\top \otimes \m{I}_d) |\m{S}^{-1}| \m{x}_c(0) \nonumber\\
&= (\m{1}_n \otimes \m{I}_d) \m{P} \sum_{i=1}^n |\m{S}^{-1}_i| \m{S}_i \m{x}_{i}(0)\nonumber\\
&= \m{1}_n \otimes \left(\m{P} \sum_{i=1}^n \text{sign}(\m{S}_i) \m{x}_{i}(0) \right).
\end{align}
Thus, $\lim_{t\to +\infty} \m{S} \m{x}(t) = \m{1}_n \otimes \m{x}^a,$
where
\begin{align} \label{eq:system_center}
\m{x}^a = \left(\sum_{i=1}^n |\m{S}_i^{-1}|\right)^{-1} \sum_{i=1}^n \text{sign}(\m{S}_i) \m{x}_{i}(t).
\end{align}
Therefore, the system \eqref{eq:MSC_x} asymptotically achieves a matrix-scaled consensus.  Because $\dot{\m{x}}^a(t) = \left(\sum_{i=1}^n |\m{S}_i^{-1}|\right)^{-1} (\m{1}_n^\top \otimes \m{I}_d) \bar{\m{L}} \m{S}\m{x}(t) = \m{0}_d,$
and which shows that ${\m{x}}^a(t)$ is time-invariant. We can now state the main theorem of this section.
\begin{Theorem} \label{thm:1} Suppose that $\mc{G}$ is undirected and connected. Under the matrix-scaled consensus algorithm \eqref{eq:MSC_agent_i}, $\m{x}(t) \to \m{S}^{-1} (\m{1}_n\otimes \m{x}^a)$ as $t\to +\infty$.
\end{Theorem}

Below, another proof of Theorem~\ref{thm:1} will be given based on Barbalat's lemma. 

\begin{proof}
Consider the function $V(\m{x}_c) = \m{x}_c^\top\bar{\m{L}}\m{x}_c$ which is positive definite with regard to $\bar{\m{L}}\m{x}_c$ and continuously differentiable. Moreover,
\begin{align*}
\dot{V} =-\m{x}^\top_c \bar{\m{L}} (|\m{S}| + |\m{S}|^\top) \bar{\m{L}} \m{x}_c.
\end{align*}
Since $|\m{S}| + |\m{S}|^\top$ is symmetric positive definite, it follows that $\dot{V}\leq 0$. It follows that $\lim_{t\to +\infty} V$ exists and is finite and $\bar{\m{L}}\m{x}_c$ is bounded.  Thus, $\ddot{V} = 2\m{x}^\top_c \bar{\m{L}} (|\m{S}| + |\m{S}|^\top) \bar{\m{L}}|\m{S}| \bar{\m{L}} \m{x}_c$ is also bounded. It follows from Barbalat's lemma \cite{Slotine1991applied} that $\lim_{t\to+\infty} \dot{V} = 0$. Thus, $\m{x}_c(t) \to \text{im}(\m{1}_n\otimes \m{I}_d)$ as $t\to+\infty$. Since $\m{x}_c = \m{S}\m{x}$, there holds
\begin{align*}
(\m{1}_n^\top\otimes \m{I}_d) |\m{S}^{-1}| \dot{\m{x}}_c = - (\m{1}_n^\top\otimes \m{I}_d) \bar{\m{L}} \m{x}_c = \m{0}_d.
\end{align*}
This means $(\m{1}_n^\top \otimes \m{I}_d) |\m{S}^{-1}| \m{x}_c(t) = (\m{1}_n^\top \otimes \m{I}_d) |\m{S}^{-1}|\m{x}_c(0)=\sum_{i=1}^n|\m{S}^{-1}|\m{S}_i\m{x}_i(0)=\sum_{i=1}^n\text{sign}(\m{S}_i)\m{x}_i(0),~\forall t\ge 0$. From $(\m{1}_n^\top \otimes \m{I}_d) |\m{S}^{-1}| (\m{1}_n\otimes \m{x}^*_c) = \Big(\sum_{i=1}^n|\m{S}_i^{-1}| \Big) \m{x}^*_c= \sum_{i=1}^n\text{sign}(\m{S}_i)\m{x}_i(0)$, it follows that $\m{x}^*_c = \m{P} \sum_{i=1}^n\text{sign}(\m{S}_i)\m{x}_i(0) = \m{x}^a$. Thus, $\m{x}_{ci}(t) \to \m{x}^a$, $\m{x}_{i}(t) \to \m{S}_i^{-1} \m{x}^a$ as $t\to +\infty$, or i.e., the system \eqref{eq:MSC_x} globally asymptotically achieves a matrix-scaled consensus.
\end{proof}
\section{Matrix-scaled consensus of double-integrator agents}
\label{sec:4}
\subsection{Proposed consensus laws}
This section studies the matrix-scaled consensus algorithm for a system of double integrators modeled by
\begin{subequations}
\label{eq:double_integrator}
\begin{align}
\dot{\m{x}}_{i}^1 &= \m{x}_{i}^2, \\
\dot{\m{x}}_{i}^2 &= \m{u}_i,~i=1, \ldots, n, 
\end{align}
\end{subequations}
where $\m{x}_{i}^1,~\m{x}_{i}^2 \in \mb{R}^{d}$ are states of agent $i$, and $\m{u}_{i} \in \mb{R}^{d}$ is its control input. Let $\m{x}_i = \text{vec}(\m{x}_{i}^1,\m{x}_{i}^2)$, $\m{x}^1 = \text{vec}(\m{x}_1^1,\ldots,\m{x}_n^1)$, and $\m{x}^2 = \text{vec}(\m{x}_1^2,\ldots,\m{x}_n^2)$. 
The objective is to make the agents' states $\m{x}^1_i$ to asymptotically achieve a matrix scaled consensus, i.e., to make $\m{x}=\text{vec}(\m{x}^1,\m{x}^2)$  asymptotically converge to the set
\begin{align*}
\mc{A}' = \{ \m{x} \in \mb{R}^{2dn}|~ \m{S}_1 \m{x}_{1}^1 =\m{S}_2 \m{x}_{2}^1 = \ldots = \m{S}_n \m{x}_{n}^1,~\m{x}^2 = \m{0}_{dn}\}
\end{align*}

The following consensus law is proposed to achieve the matrix-scaled consensus:
\begin{align} 
\m{u}_i &= -\text{sign}(\m{S}_i) \sum_{j\in \mc{N}_i} (\m{S}_i \m{x}_i^1 - \m{S}_j \m{x}_j^1) - \alpha \m{x}_i^2, \label{eq:MSC_double1a}
\end{align}
where $\alpha>0$ is a control gain. The $n$-agent system under \eqref{eq:MSC_double1a} is given as follows
\begin{subequations}
\label{eq:MSC_double1b}
\begin{align} 
\dot{\m{x}}^1&= \m{x}^2,\\
\dot{\m{x}}^{2}&= - (\text{sign}(\m{S}_i))\otimes \m{I}_d) \bar{\m{L}} \m{S} \m{x}^1 - \alpha \m{x}^2. 
\end{align}
\end{subequations}

\subsection{Stability analysis}
\label{subsection:4b}
The behavior of the system \eqref{eq:MSC_double1b} is given in the following theorem. 
\begin{Theorem} Suppose that $\mc{G}$ is undirected and connected. Under the consensus law \eqref{eq:MSC_double1a}, $\m{x}(t)$ asymptotically converges to a point in $\mc{A}'$.
\end{Theorem}
\begin{proof}
Let $\m{x}^1_c = \m{S} \m{x}^1$ and $\m{x}^2_c = \m{S} \m{x}^2$, we can rewrite the system \eqref{eq:MSC_double1b} as follows
\begin{align} \label{eq:MSC_double1c}
\begin{bmatrix}
\dot{\m{x}}_c^1\\
\dot{\m{x}}_c^2
\end{bmatrix}
&= 
\begin{bmatrix}
\m{0}_{dn} & \m{I}_{dn}\\
- |\m{S}|\bar{\m{L}} & - \alpha \m{I}_{dn}
\end{bmatrix} 
\begin{bmatrix}
\m{x}^1_c \\ \m{x}^2_c
\end{bmatrix} = \m{N}_1 \begin{bmatrix}
\m{x}^1_c \\ \m{x}^2_c
\end{bmatrix}.
\end{align}
Substituting $\m{\Theta} = |\m{S}|\bar{\m{L}} = \m{W}\m{J}\m{W}^{-1}$ into the characteristic equation $\text{det}(s\m{I}_{2dn} - \m{N}_1)=0$, one has
\begin{align}
\text{det}(s^2 \m{I}_{dn} + \alpha s \m{I}_{dn} + \m{W}\m{J}\m{W}^{-1}) &= 0, \nonumber\\
\text{or, equivalently } \prod_{k=1}^{dn} (s^2 + \alpha s + \mu_k) = & 0,\qquad\quad
\end{align}
where $\mu_1 = \ldots = \mu_d =0$ and $\mb{C} \ni \mu_k = a_k + j b_k \ne 0$, $\forall k= d+1, \ldots, {dn}$. Based on Lemma~\ref{lem:complex_quadratic_equation}, each polynomial $s^2 + \alpha s + \mu_k$ is Hurwitz if and only if $a_k > {b_k^2}/{\alpha^2}$. Thus, by choosing $\alpha$ so that
\begin{equation} \label{eq:ConvergenceCondition1}
\min_{k=d+1,\ldots,dn}\text{Re}(\mu_k) > \frac{(\max_{k=d+1,\ldots,dn}\text{Im}(\mu_k))^2}{\alpha^2},
\end{equation}
the matrix $\m{N}_1$ has $d$ zero eigenvalues and $2dn-d$ eigenvalues with negative real parts. The right and left eigenvectors of $\m{N}_1$  corresponding to the zero eigenvalues, are columns and rows of the matrices 
$\begin{bmatrix}
\m{1}_{n} \otimes \m{I}_d\\
\m{0}_n \otimes \m{I}_d
\end{bmatrix}$, and $\m{P}\begin{bmatrix}
(\m{1}_{n}^\top \otimes \m{I}_d)|\m{S}|^{-1}& {\alpha}^{-1}(\m{1}_{n}^\top \otimes \m{I}_d )|\m{S}|^{-1}
\end{bmatrix}.$
Thus,
\begin{align*}
\lim_{t\to +\infty} \begin{bmatrix}
\m{x}^1_c \\ \m{x}^2_c
\end{bmatrix} &= \lim_{t\to +\infty} \text{exp}({\m{N}_1t}) \begin{bmatrix}
\m{x}^1_c(0) \\ \m{x}^2_c(0)
\end{bmatrix} \\
&= \begin{bmatrix}
\m{1}_n\otimes \Big(\m{x}^{a1} + \frac{\m{x}^{a2}}{\alpha} \Big)\\ \m{0}_{dn}
\end{bmatrix}
\end{align*}
where $\m{x}^{al} = \Big(\sum_{i=1}^n |\m{S}_i^{-1}|\Big)^{-1} \sum_{i=1}^n \text{sign}(\m{S}_i)\m{x}_{i}^l(0),~l=1, 2.$
This implies that $\lim_{t\to +\infty} \m{S}_i\m{x}_i^1(t) = \m{x}^{a1} + \alpha^{-1} \m{x}^{a2}$ and $\lim_{t\to +\infty} \m{x}_i^2(t) = \m{0}_d$, $=1,\ldots, n$.
\end{proof}
\begin{figure}[t]
\centering
\includegraphics[height=5cm]{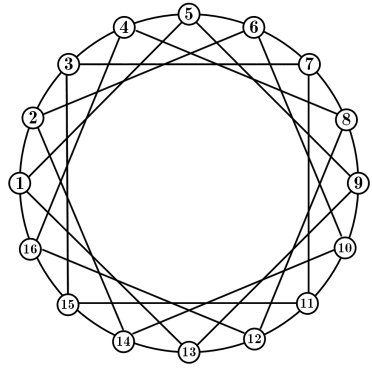}
\caption{The 16-vertex graph used in the simulations.}
\label{fig:graph}
\end{figure}
\section{Simulation results}
\label{sec:5}
Consider a system of sixteen agents having the interaction graph as depicted in Fig.~\ref{fig:graph}. We will provide some simulations to support the results in the previous sections. 

\subsection{Simulation 1: MSC of single integrators}
Let the SO(2) rotation matrix of angle $\theta$ (rad) be denoted by $\m{R}(\theta) = \begin{bmatrix}
\cos(\theta) & -\sin(\theta) \\ \sin(\theta) & \cos(\theta)
\end{bmatrix}$. Let the scaling matrices be chosen as $\m{S}_1=\ldots=\m{S}_6 = \m{R}(0)=\m{I}_2$ (positive definite), $\m{S}_7 = \ldots = \m{S}_{11} =\m{R}(\frac{2\pi}{3})$, and $\m{S}_{12} = \ldots = \m{S}_{16} = \m{R}(\frac{4\pi}{3})$ (negative definite). The initial condition $\m{x}(0)$ is randomly selected. The simulation results of~\eqref{eq:MSC_agent_i} depicted in Fig.~\ref{fig:2} show that the agents converge to three clusters, which are three vertices of an equilateral triangle. Agents with the same matrix $\m{S}_i$ converge to the same cluster. Notice that $\theta$ should be  not equal to $\frac{\pi}{2}$ so that the analysis holds.

Next, the scaling matrices are rotation matrices in $SO(3)$ so that $\m{S}_1 = \ldots = \m{S}_4$, $\m{S}_5 = \ldots = \m{S}_8$, $\m{S}_9 = \ldots = \m{S}_{12}$, and $\m{S}_{13} = \ldots = \m{S}_{16}$. The matrices $\m{S}_1$, $\m{S}_5$ are chosen to be positive definite and $\m{S}_9$, $\m{S}_{13}$ are chosen to be negative definite. Fig.~\ref{fig:3} shows that the agents converge to 4 clusters in the space. 
\begin{figure*}
\centering
\subfloat{\includegraphics[width = 0.25\linewidth]{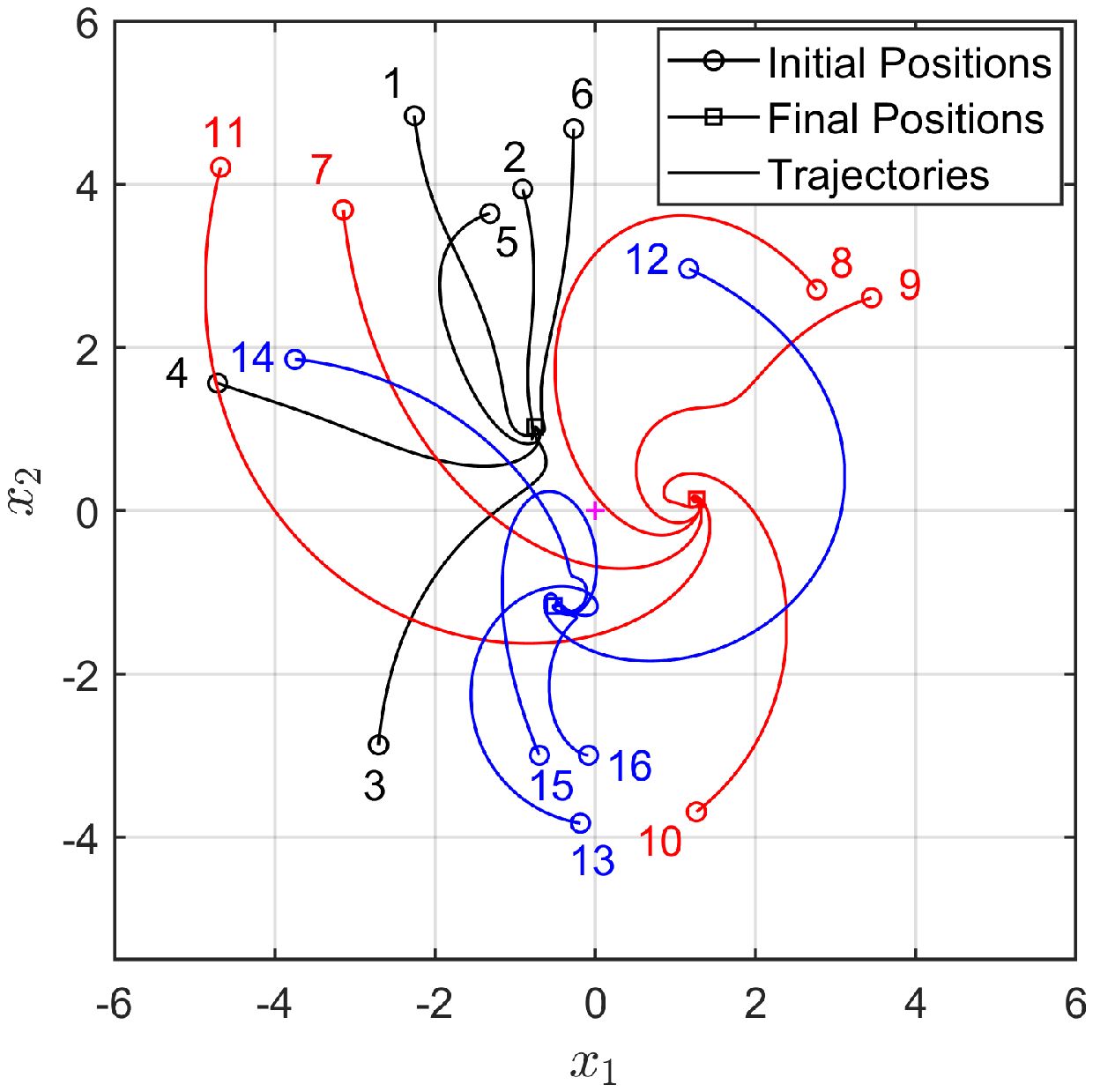}} \qquad \qquad
\subfloat{\includegraphics[width = 0.16\linewidth]{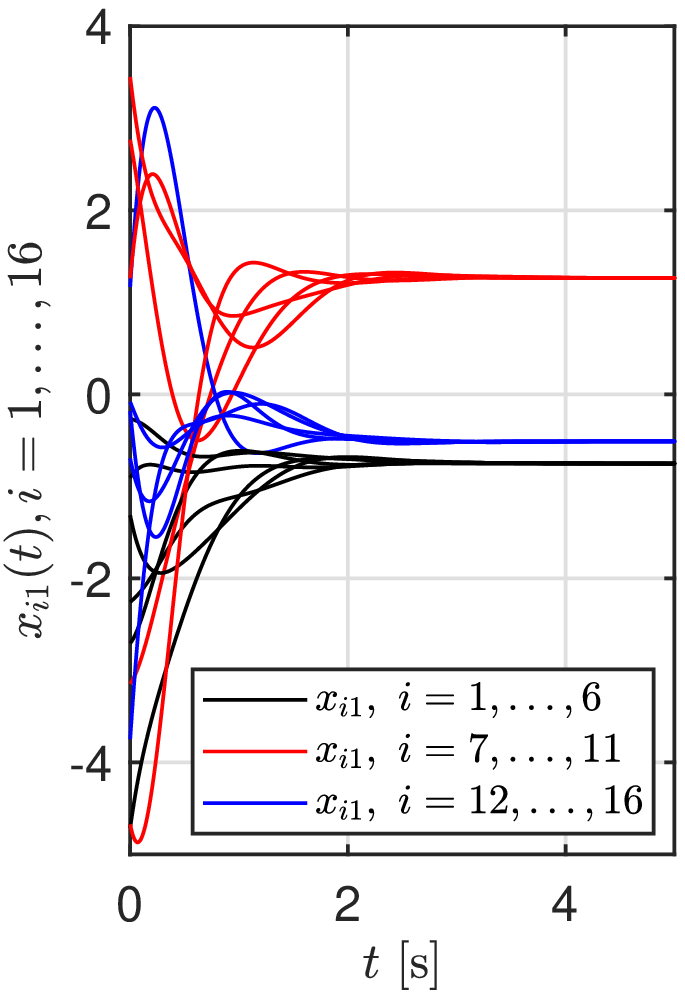}} \qquad \qquad
\subfloat{\includegraphics[width = 0.16\linewidth]{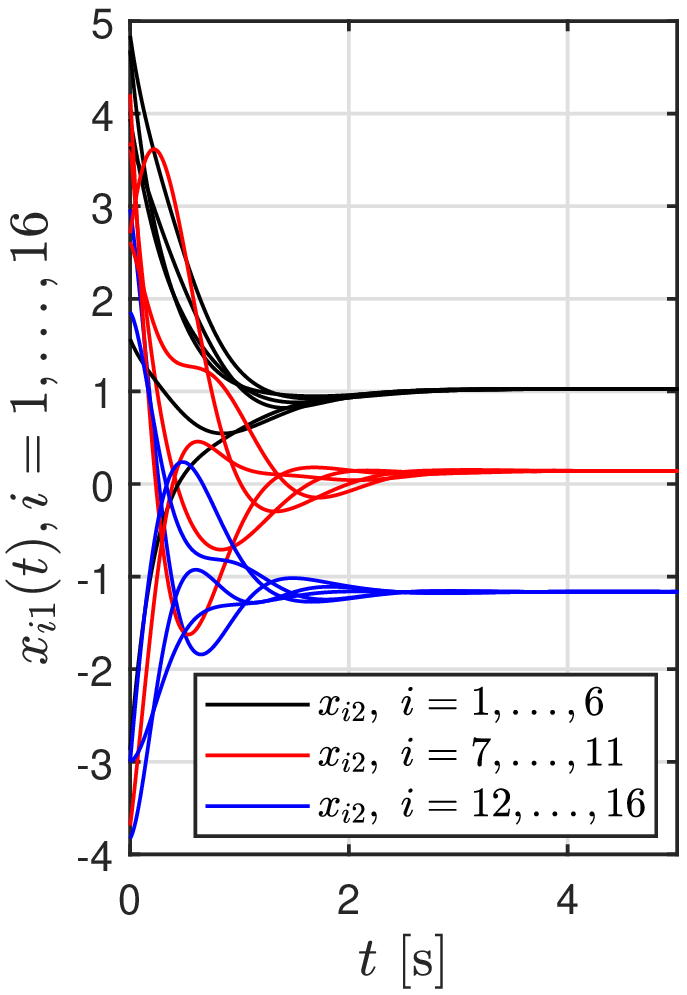}}
\caption{The agents converge to 3 clusters in the plane under the matrix-scaled consensus algorithm \eqref{eq:MSC_agent_i}.}
\label{fig:2}
\end{figure*}
\begin{figure*}
\centering
\subfloat{\includegraphics[width = 0.25\linewidth]{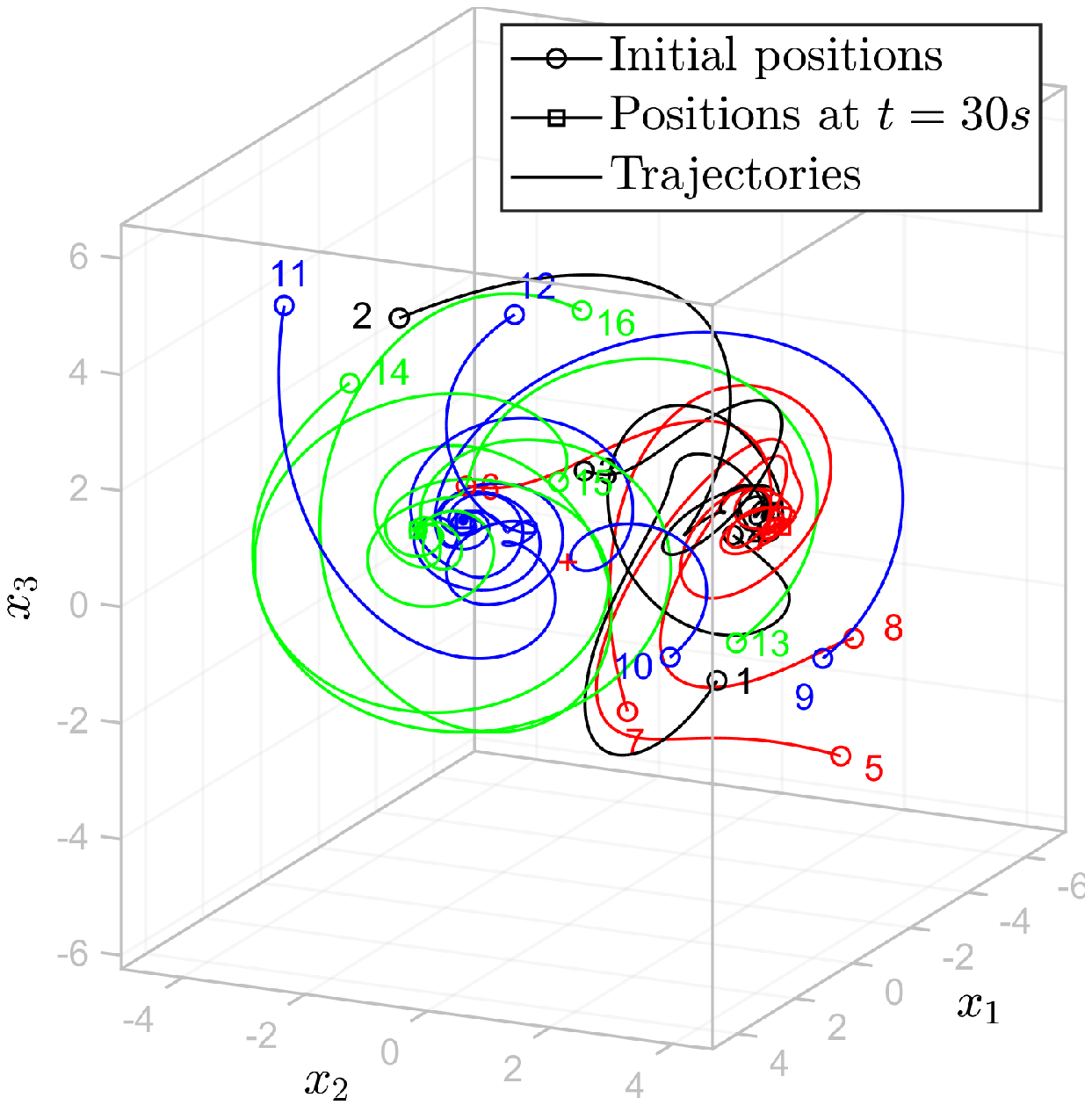}} \qquad
\subfloat{\includegraphics[width = 0.16\linewidth]{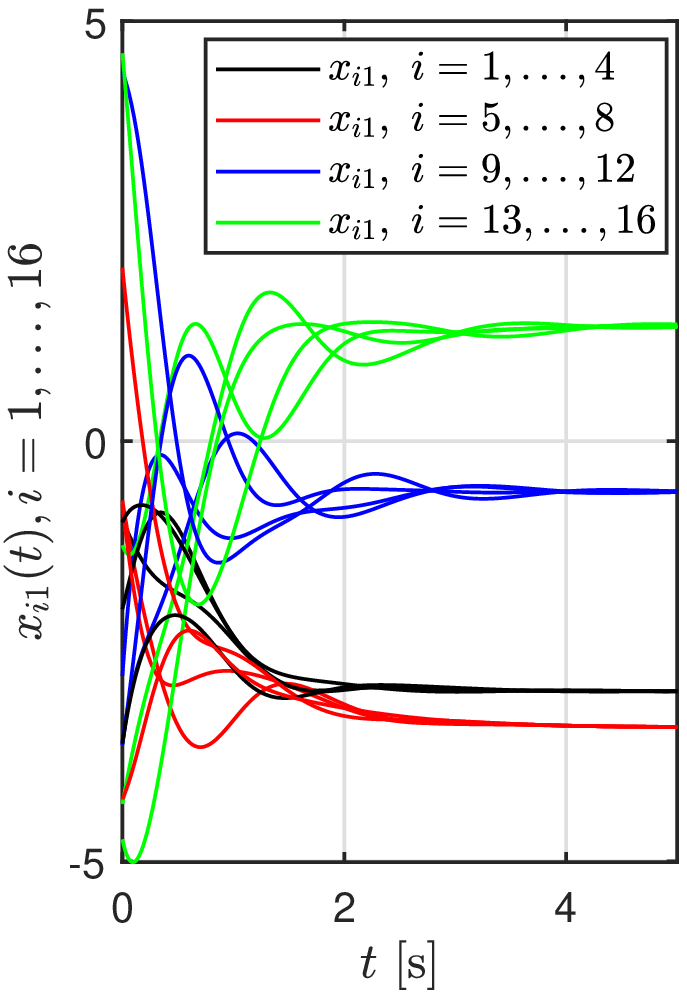}} \qquad
\subfloat{\includegraphics[width = 0.16\linewidth]{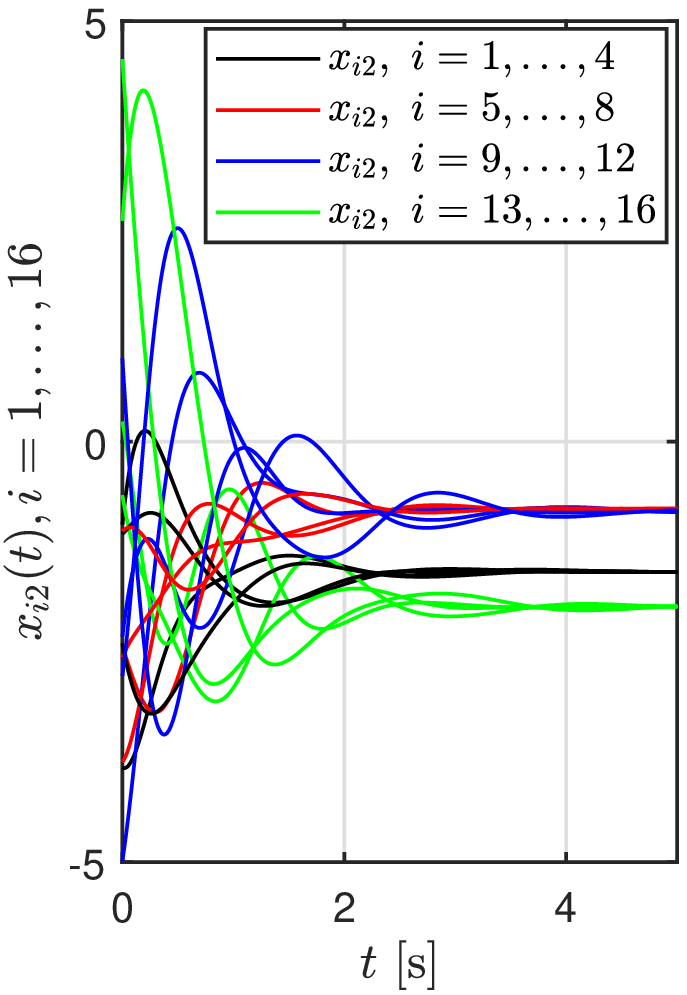}} \qquad
\subfloat{\includegraphics[width = 0.16\linewidth]{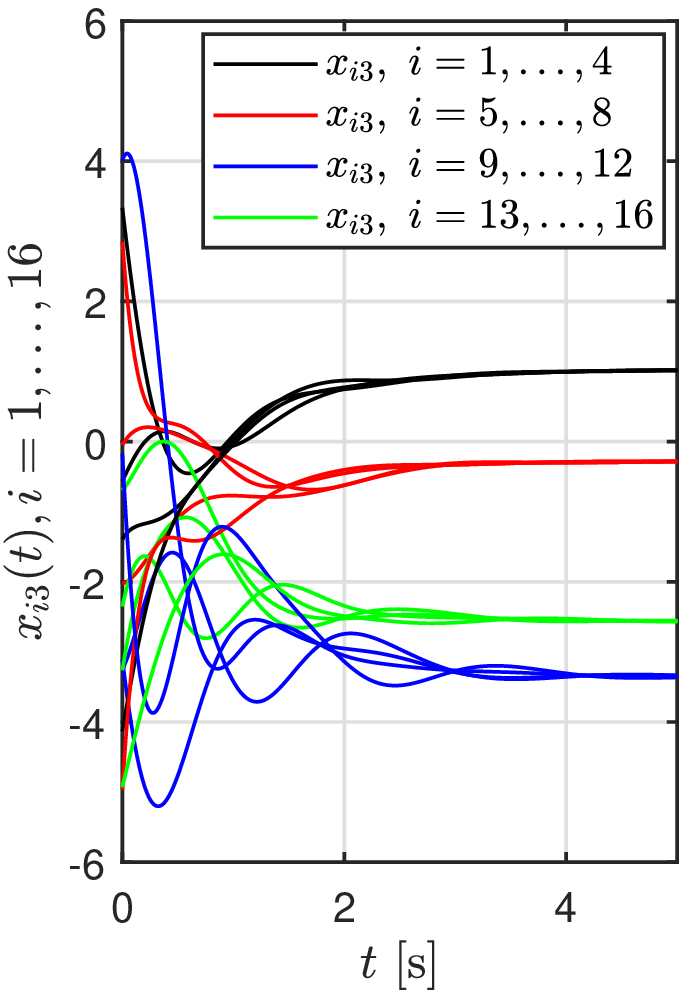}}
\caption{The agents converge to 4 clusters in three-dimensional space under the matrix-scaled consensus algorithm \eqref{eq:MSC_agent_i}.}
\label{fig:3}
\end{figure*}

\subsection{Simulation 2: MSC of double-integrators}
Next, let the agents be modeled by double-integrators. The scaling matrices are $\m{S}_i = \m{R}(\frac{\pi}{4}), i=1,\ldots, 4$, $\m{S}_i = \m{R}(\frac{3\pi}{4}),~i=5,\ldots, 8$, $\m{S}_i = \m{R}(-\frac{3\pi}{4}),~i=9,\ldots, 12$, and $\m{S}_i = \m{R}(-\frac{\pi}{4}),~i=13,\ldots, 16$. We conduct three simulations of the MSC algorithm \eqref{eq:MSC_double1a} with $\alpha = 1.8$, $1.9724$ and $3$, respectively. The trajectories of the agents, corresponding to these parameters are depicted in Fig.~\ref{fig:4}. For ${\alpha}=1.8$, the system is unstable. Correspondingly, Figs.~\ref{fig:4} (a), (d)--(e) show the state variables grow unbounded. For $\alpha = 1.9724$, $\bm{\Theta}$ has pairs of imaginary eigenvalues with the corresponding independent eigenvectors, $x_{i}^l, l=1, 2,$ are asymptotic to sinusoidal functions (see Figs.~\ref{fig:4} (b), (h)--(k)). Finally, for $\alpha = 3$, the condition \eqref{eq:ConvergenceCondition1} is satisfied. The agents converge to 4 clusters as shown in Figs.~\ref{fig:4} (c), (l)--(o).

\begin{figure*}
\subfloat[$\alpha=1.8$]{\includegraphics[width=0.33\linewidth]{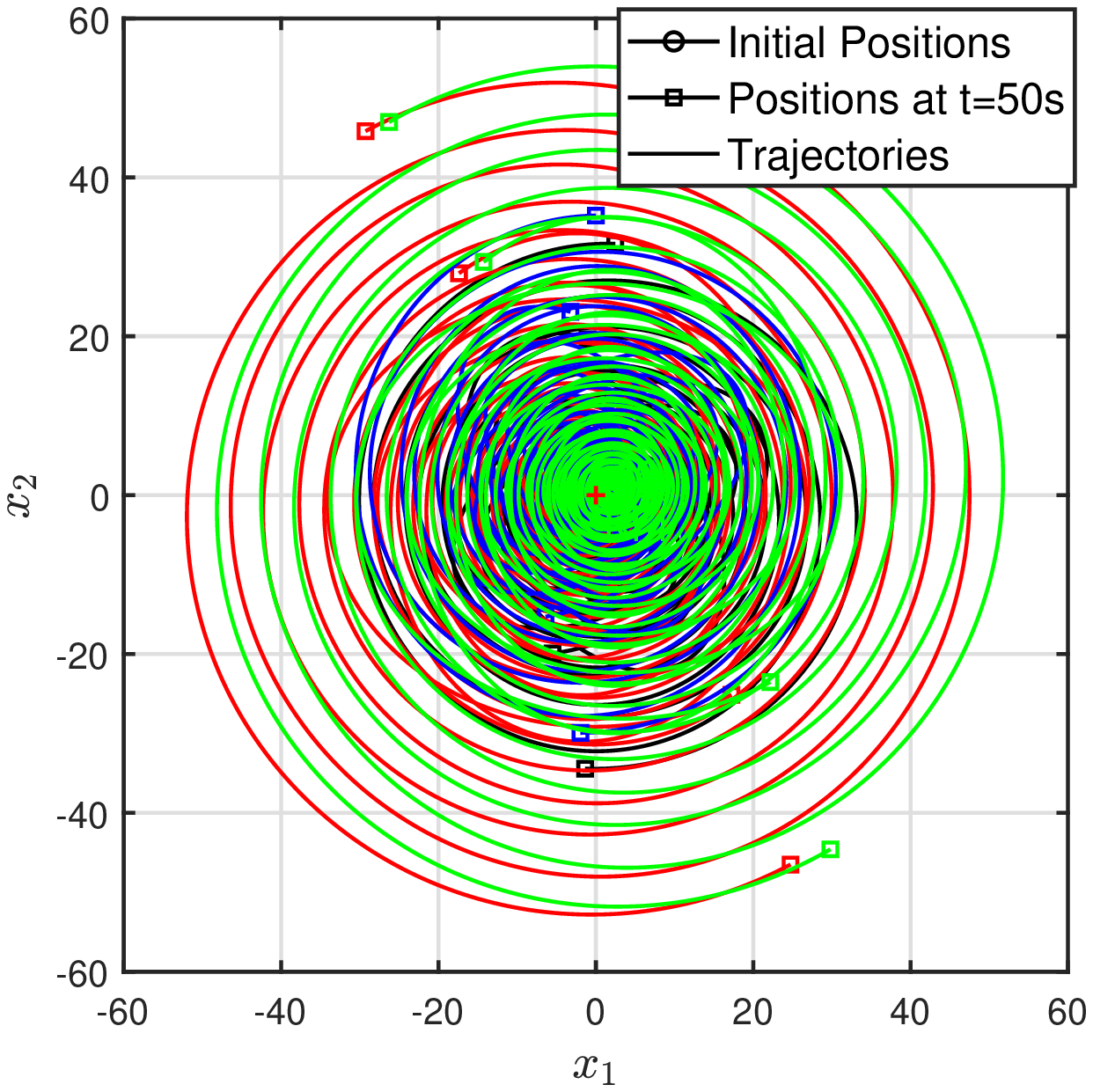}} \hfill
\subfloat[$\alpha=1.9724$]{\includegraphics[width=0.33\linewidth]{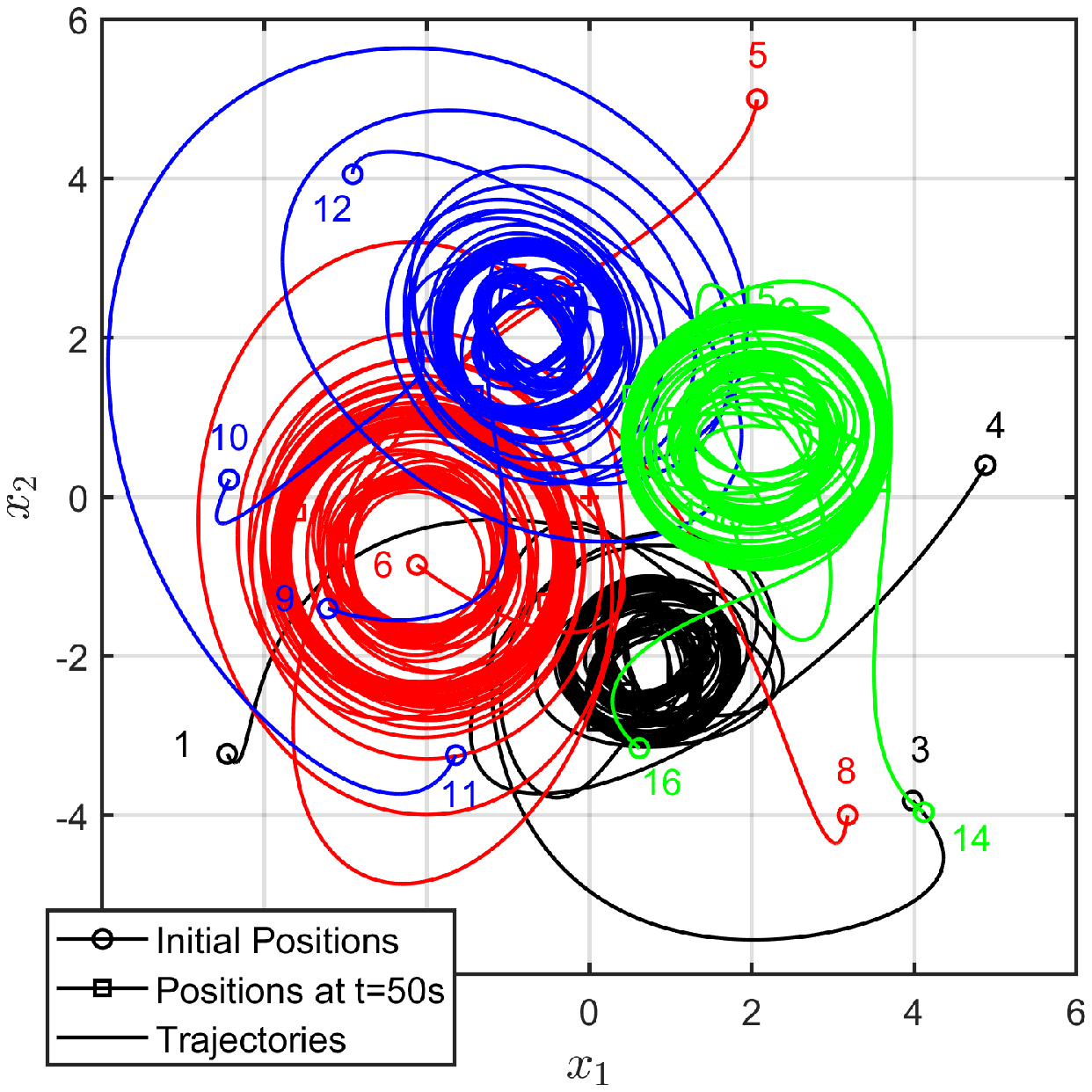}} \hfill
\subfloat[$\alpha=3$]{\includegraphics[width=0.33\linewidth]{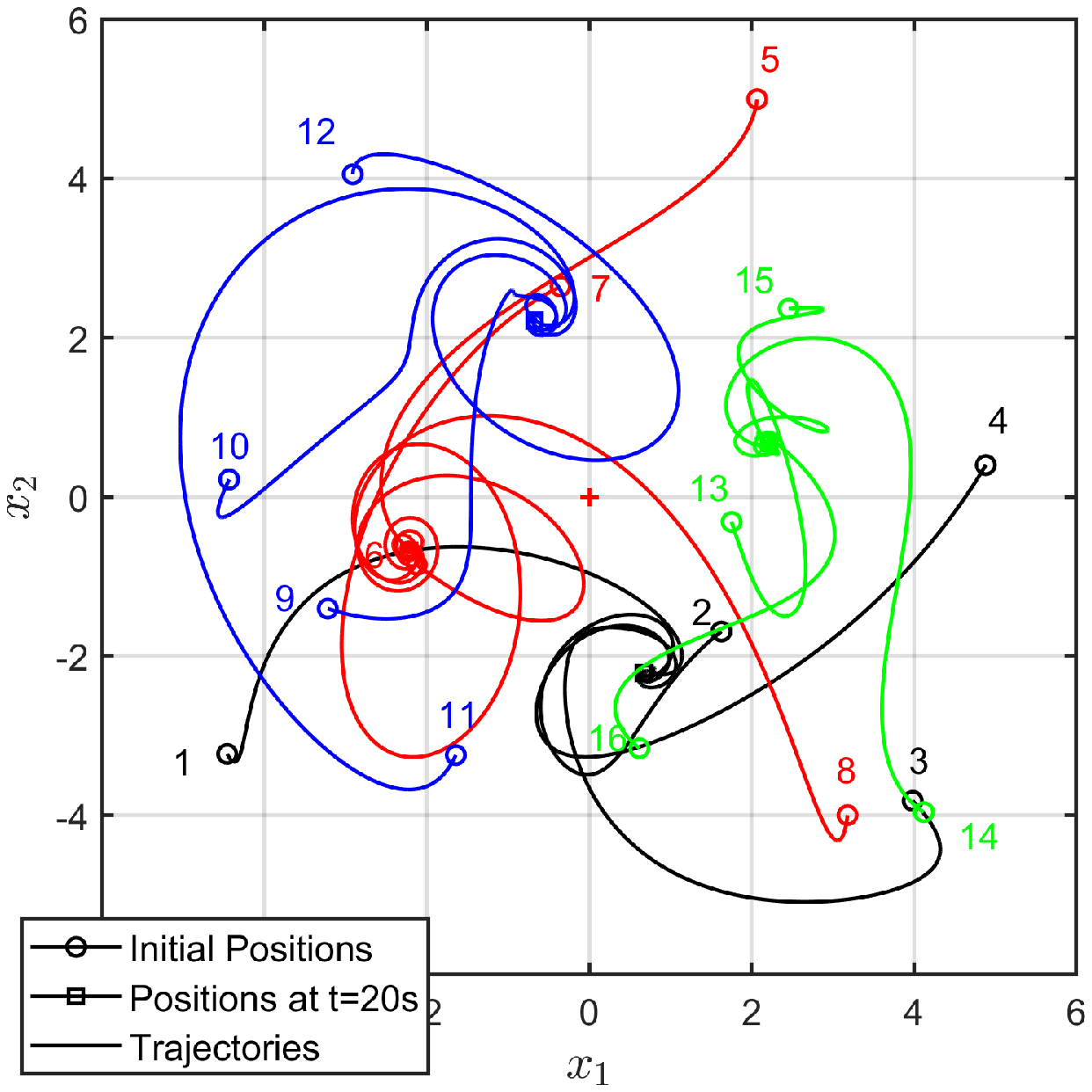}}\\
\subfloat[{$x_{i1}^1$ vs $t$ [s]}]{\includegraphics[width=0.24\linewidth]{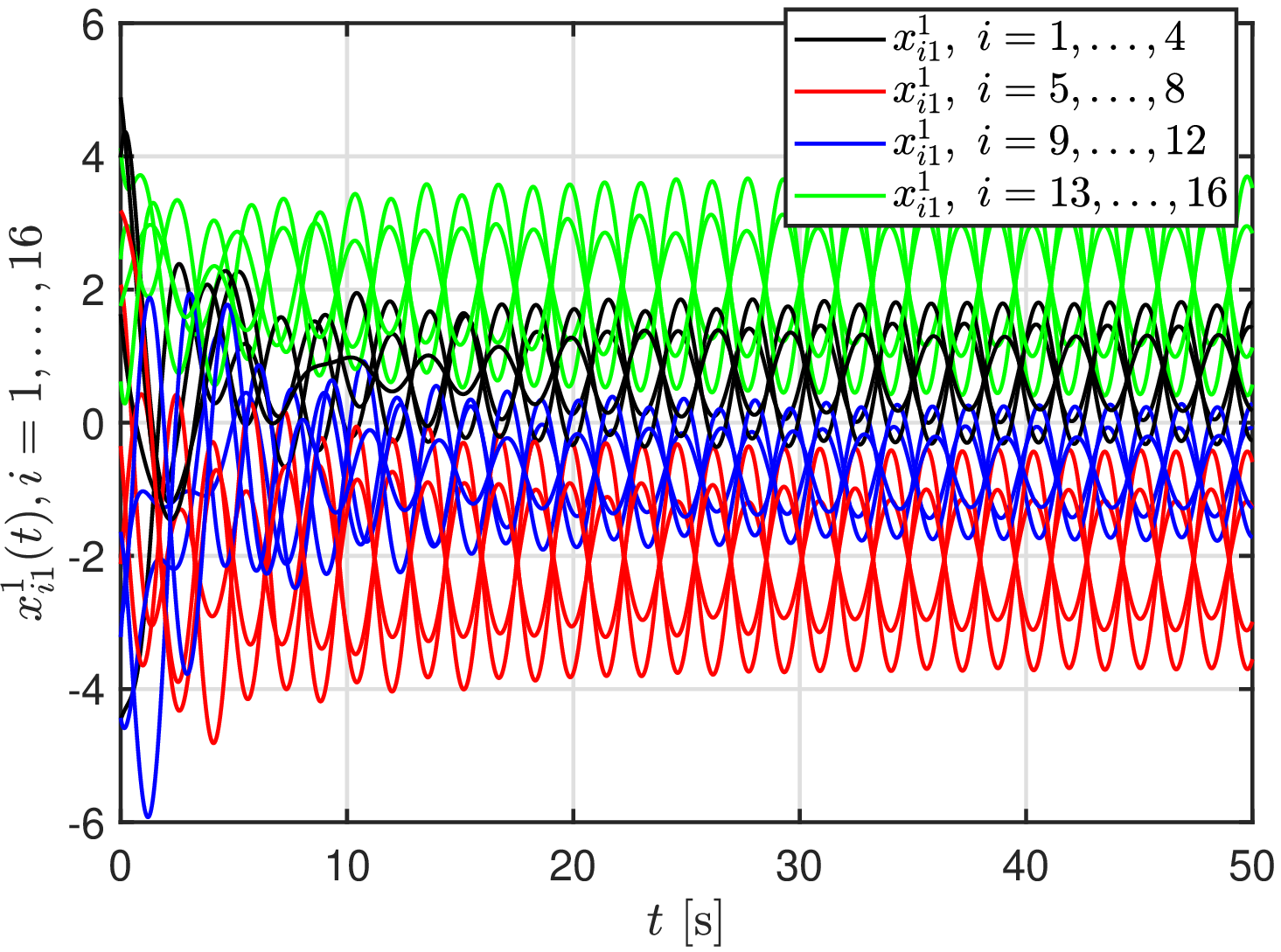}} \hfill
\subfloat[{$x_{i2}^1$ vs $t$ [s]}]{\includegraphics[width=0.24\linewidth]{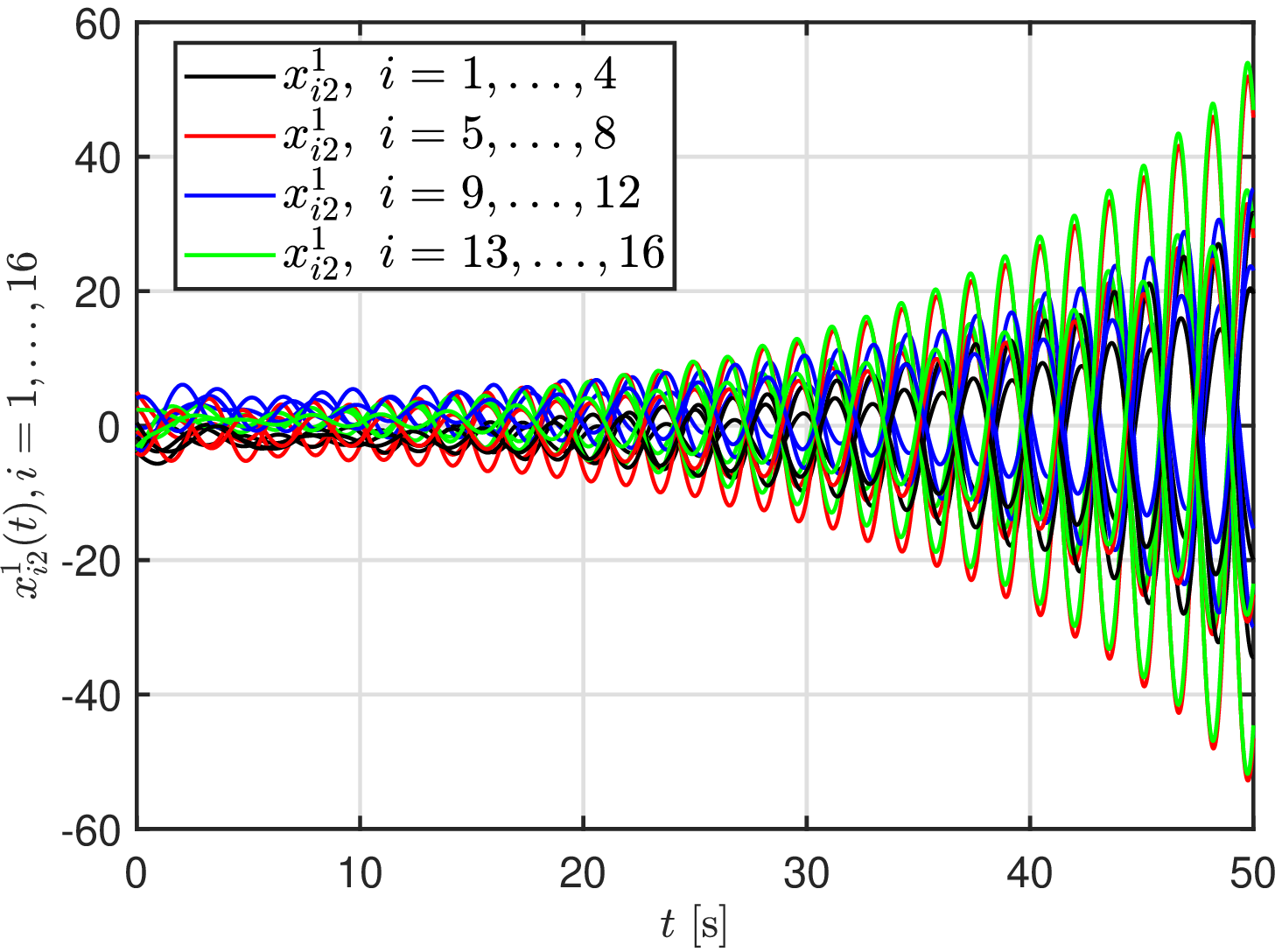}} \hfill
\subfloat[{$x_{i1}^2$ vs $t$ [s]}]{\includegraphics[width=0.24\linewidth]{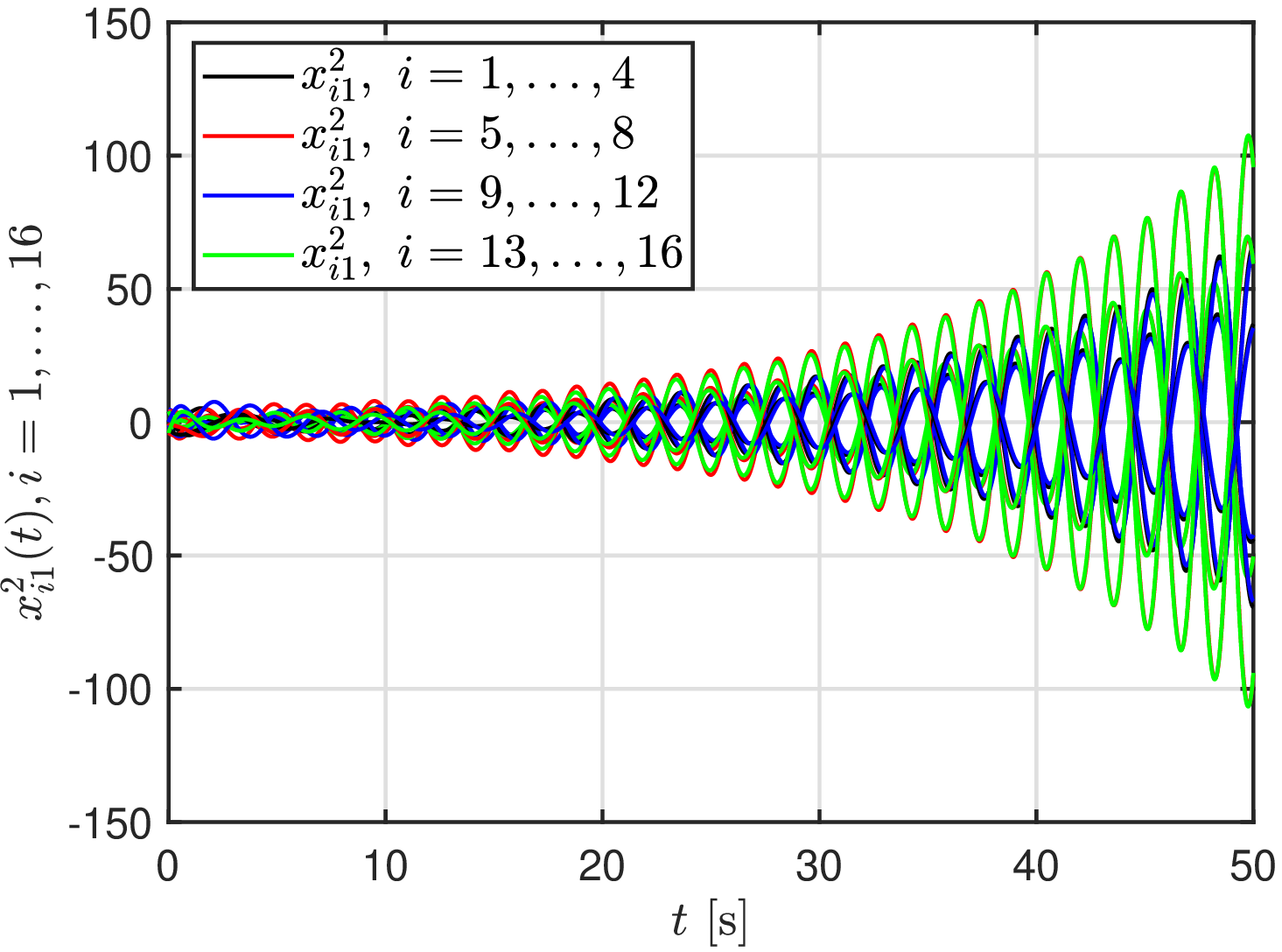}} \hfill
\subfloat[{$x_{i2}^2$ vs $t$ [s]}]{\includegraphics[width=0.24\linewidth]{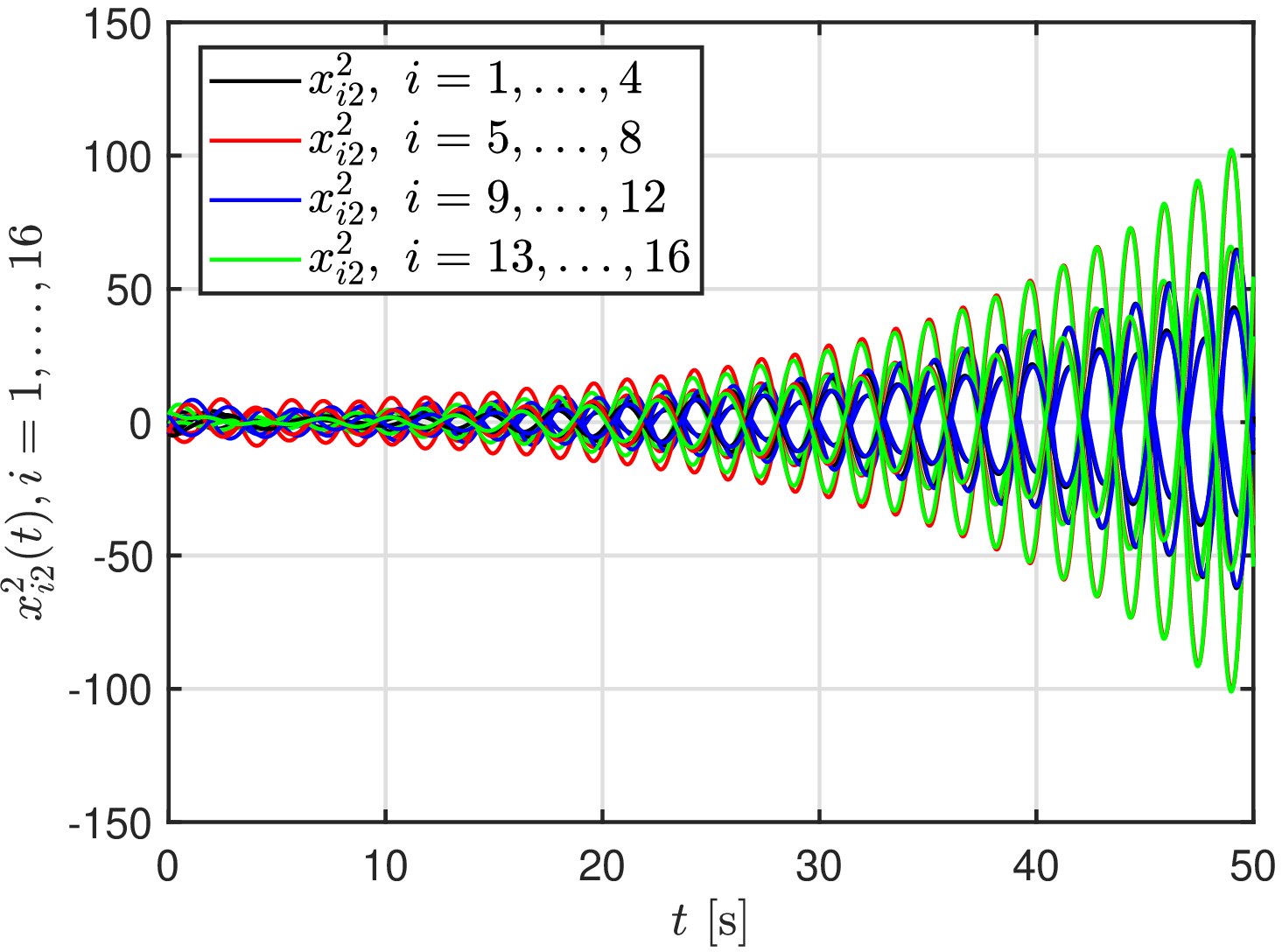}} \hfill\\
\subfloat[{$x_{i1}^1$ vs $t$ [s]}]{\includegraphics[width=0.24\linewidth]{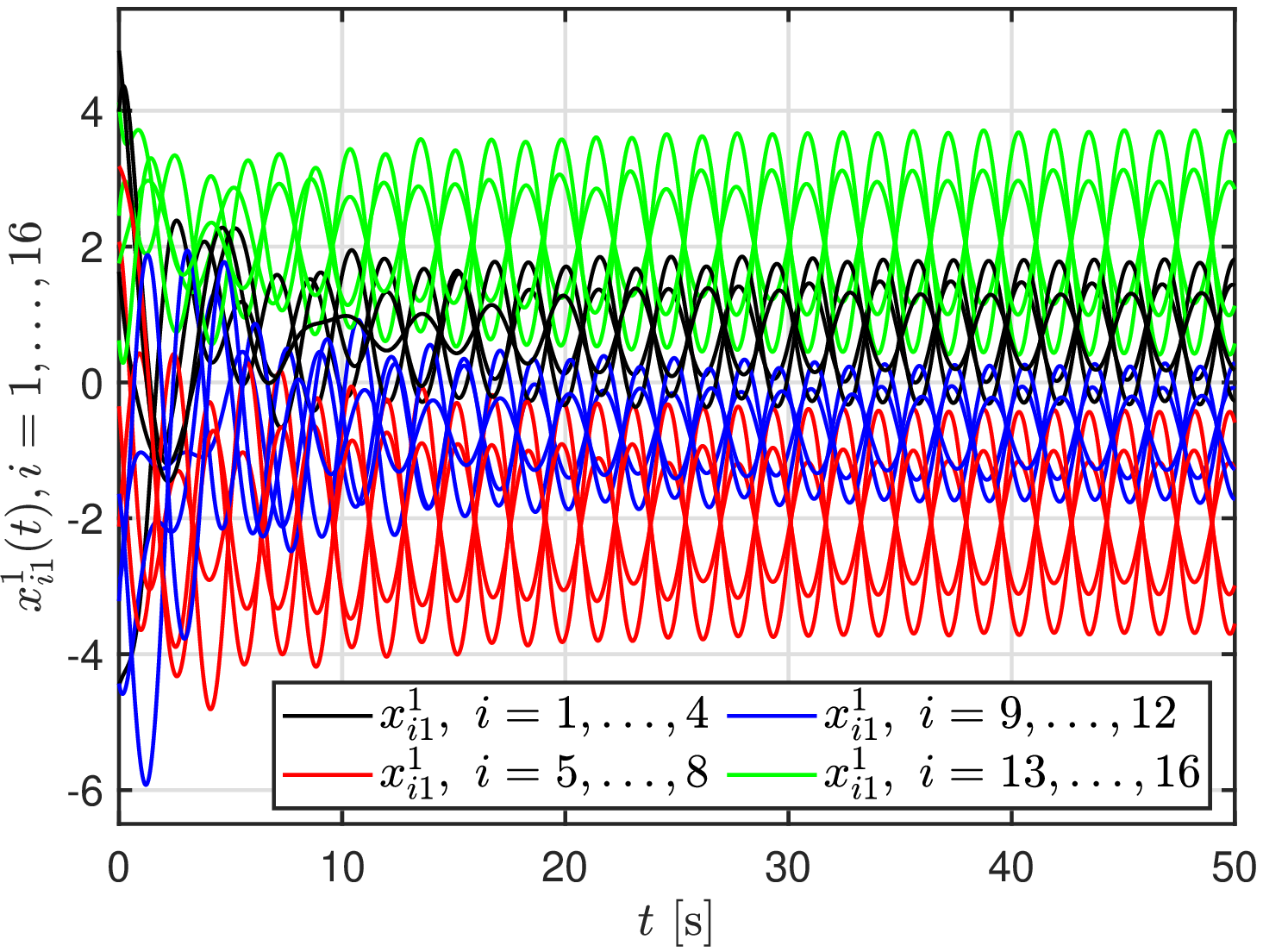}} \hfill
\subfloat[{$x_{i2}^1$ vs $t$ [s]}]{\includegraphics[width=0.24\linewidth]{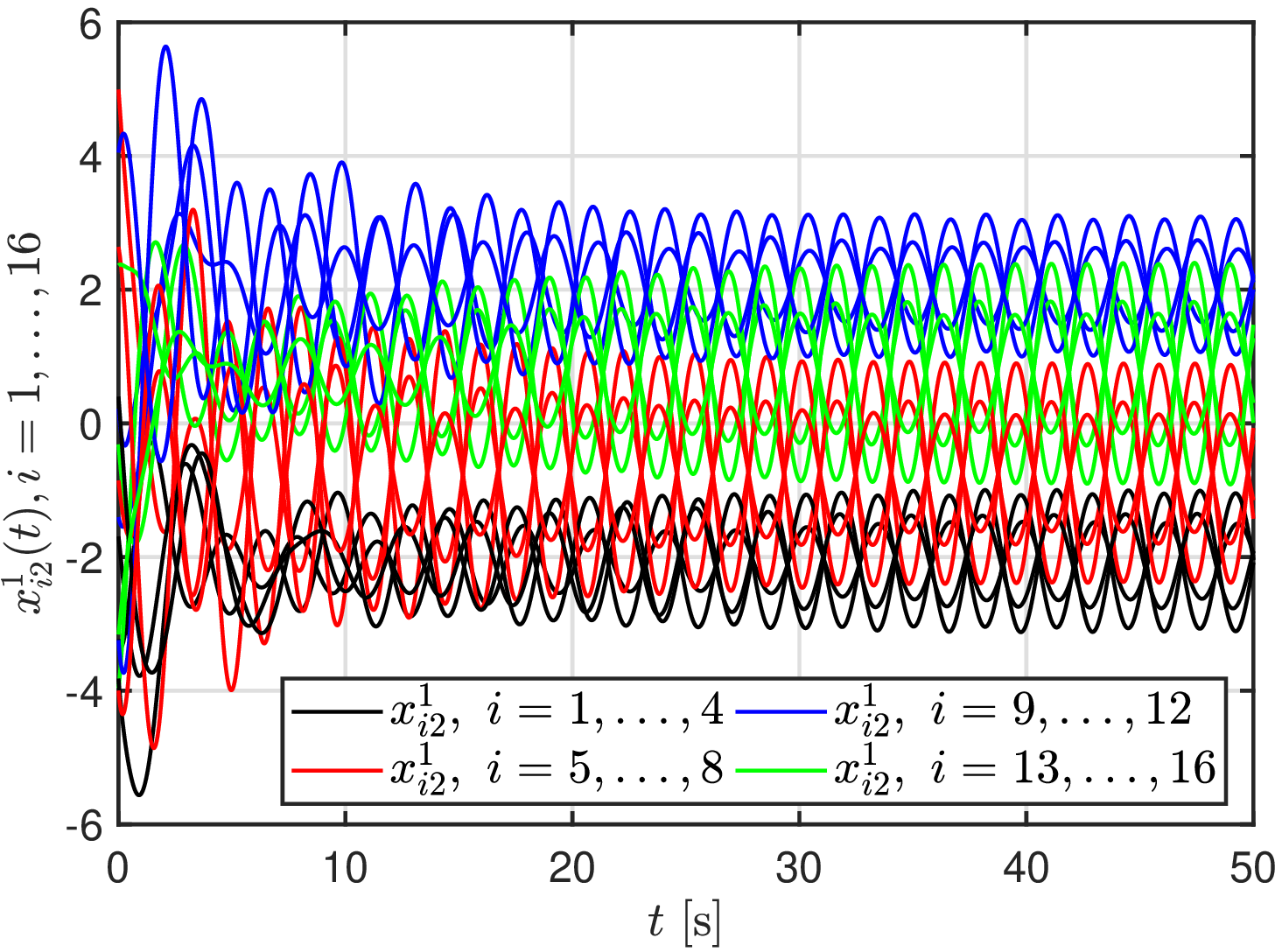}} \hfill
\subfloat[{$x_{i1}^2$ vs $t$ [s]}]{\includegraphics[width=0.24\linewidth]{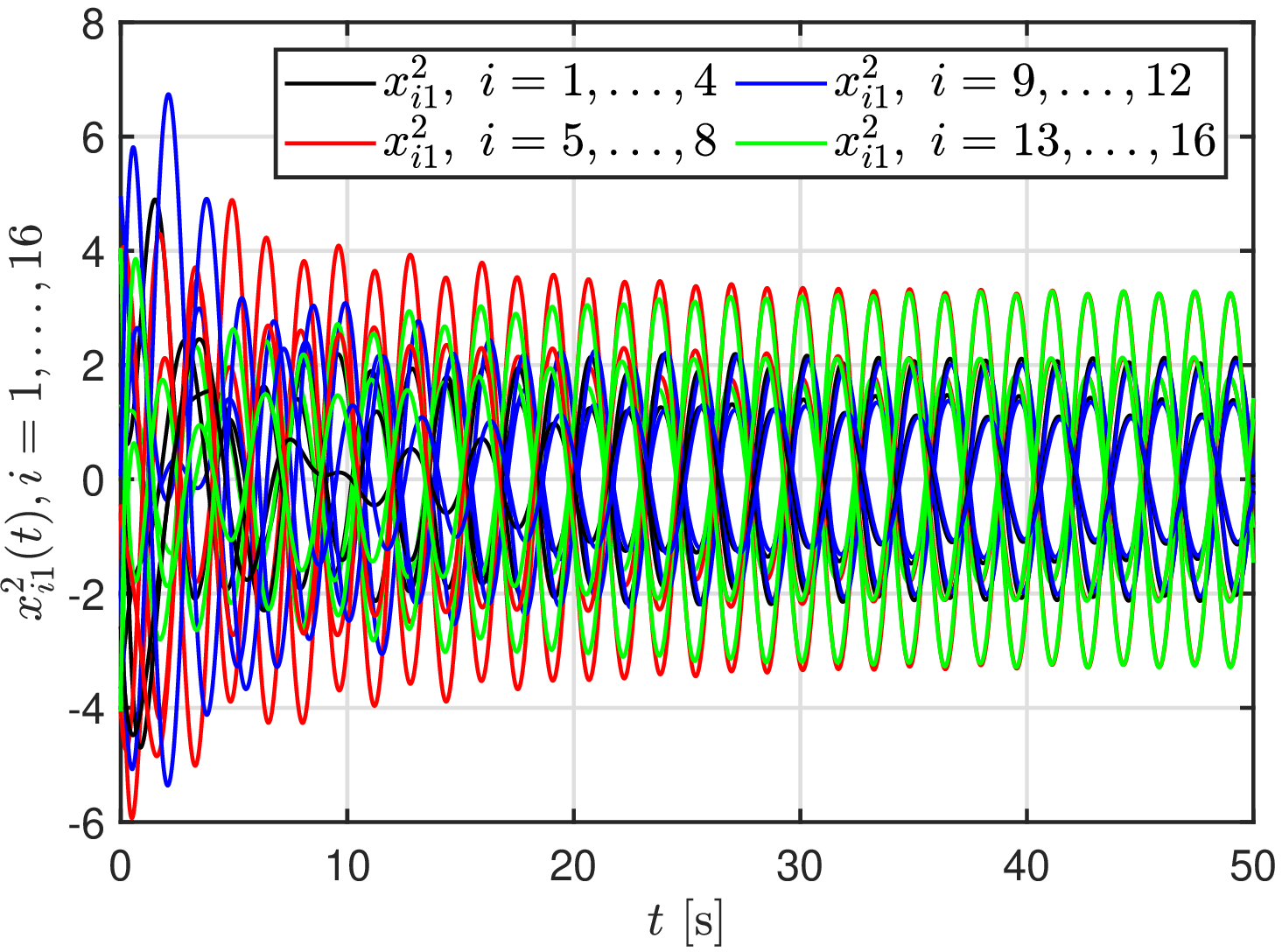}} \hfill
\subfloat[{$x_{i2}^2$ vs $t$ [s]}]{\includegraphics[width=0.24\linewidth]{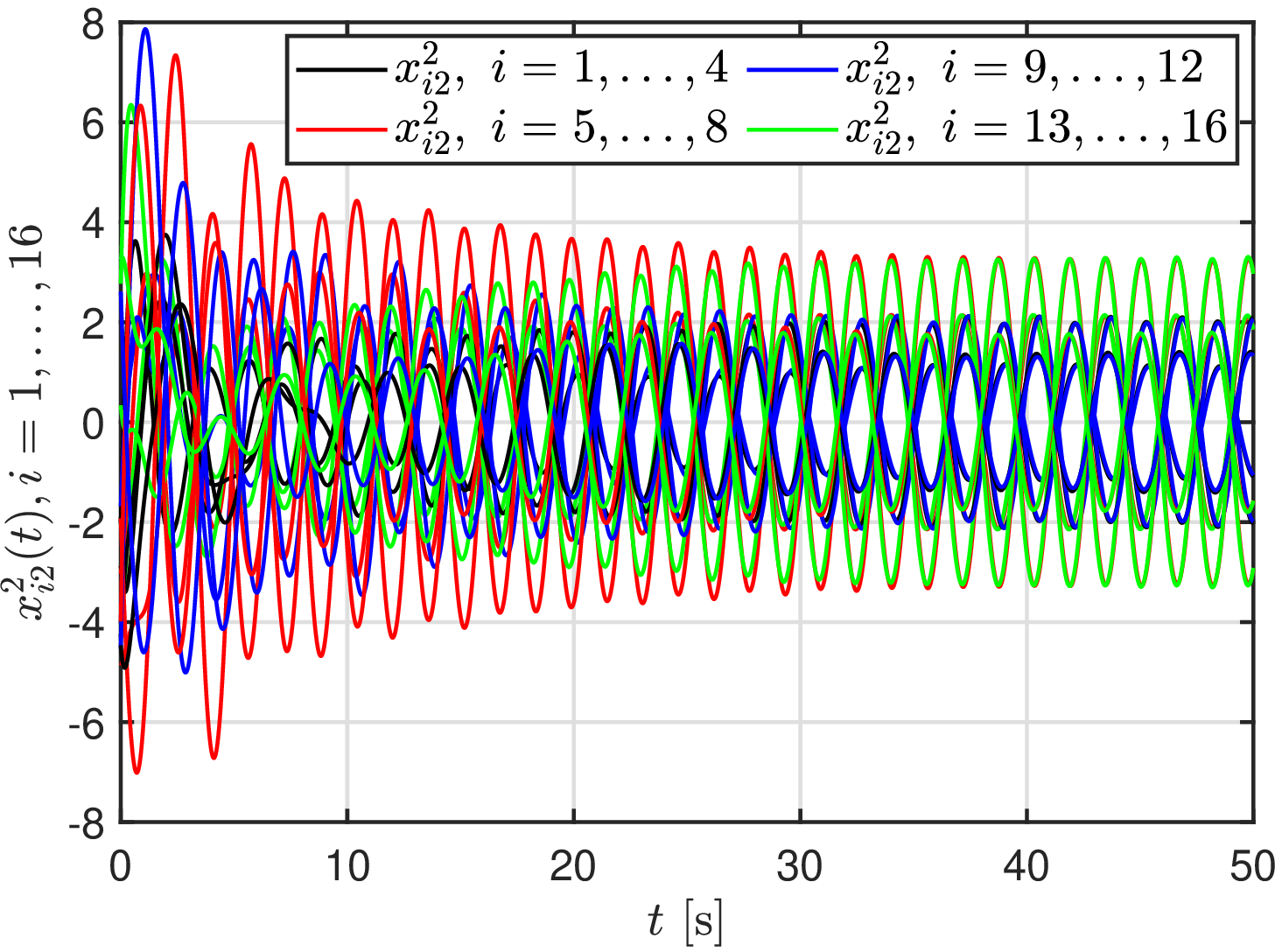}} \hfill\\
\subfloat[{$x_{i1}^1$ vs $t$ [s]}]{\includegraphics[width=0.24\linewidth]{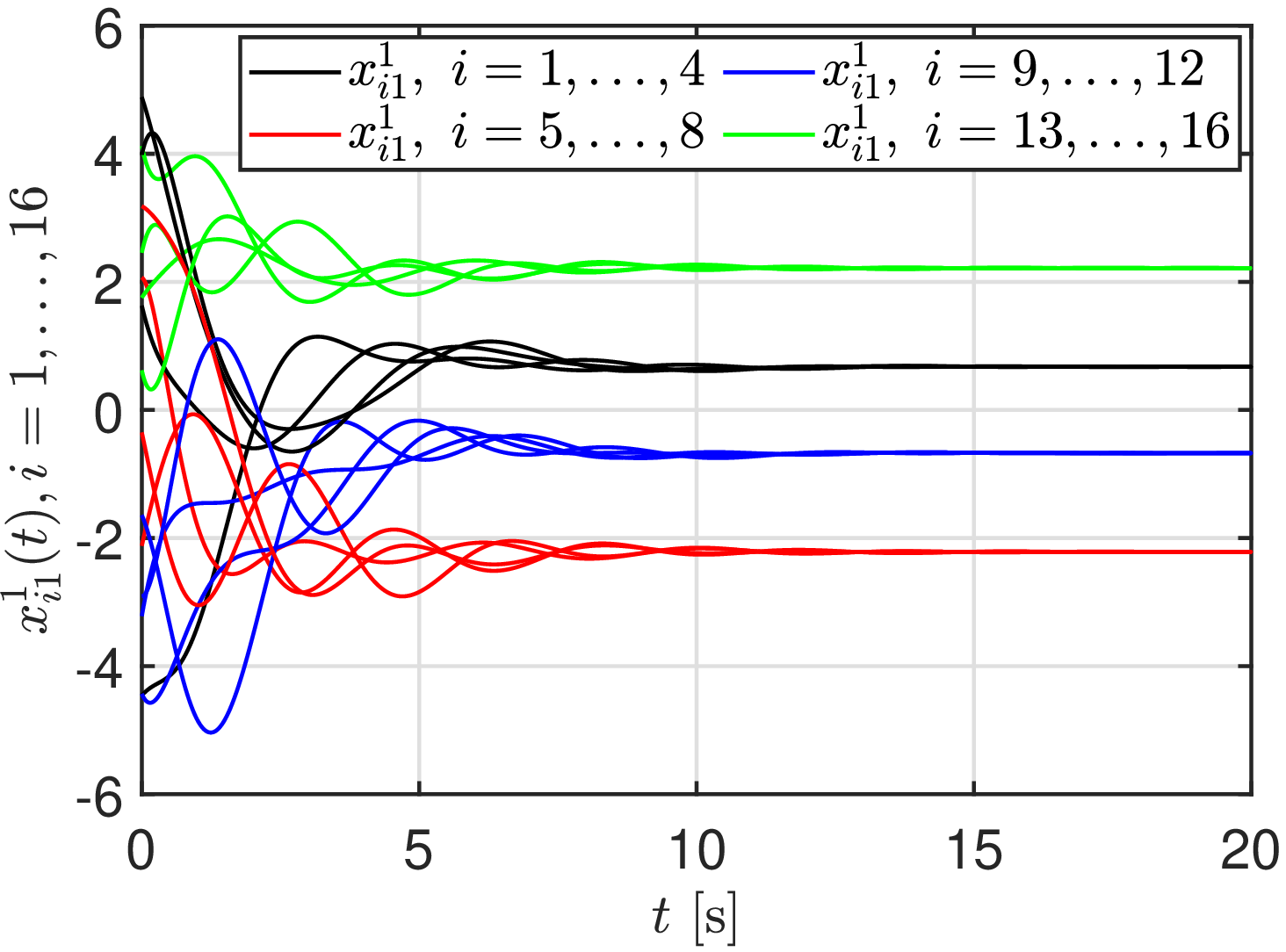}} \hfill
\subfloat[{$x_{i2}^1$ vs $t$ [s]}]{\includegraphics[width=0.24\linewidth]{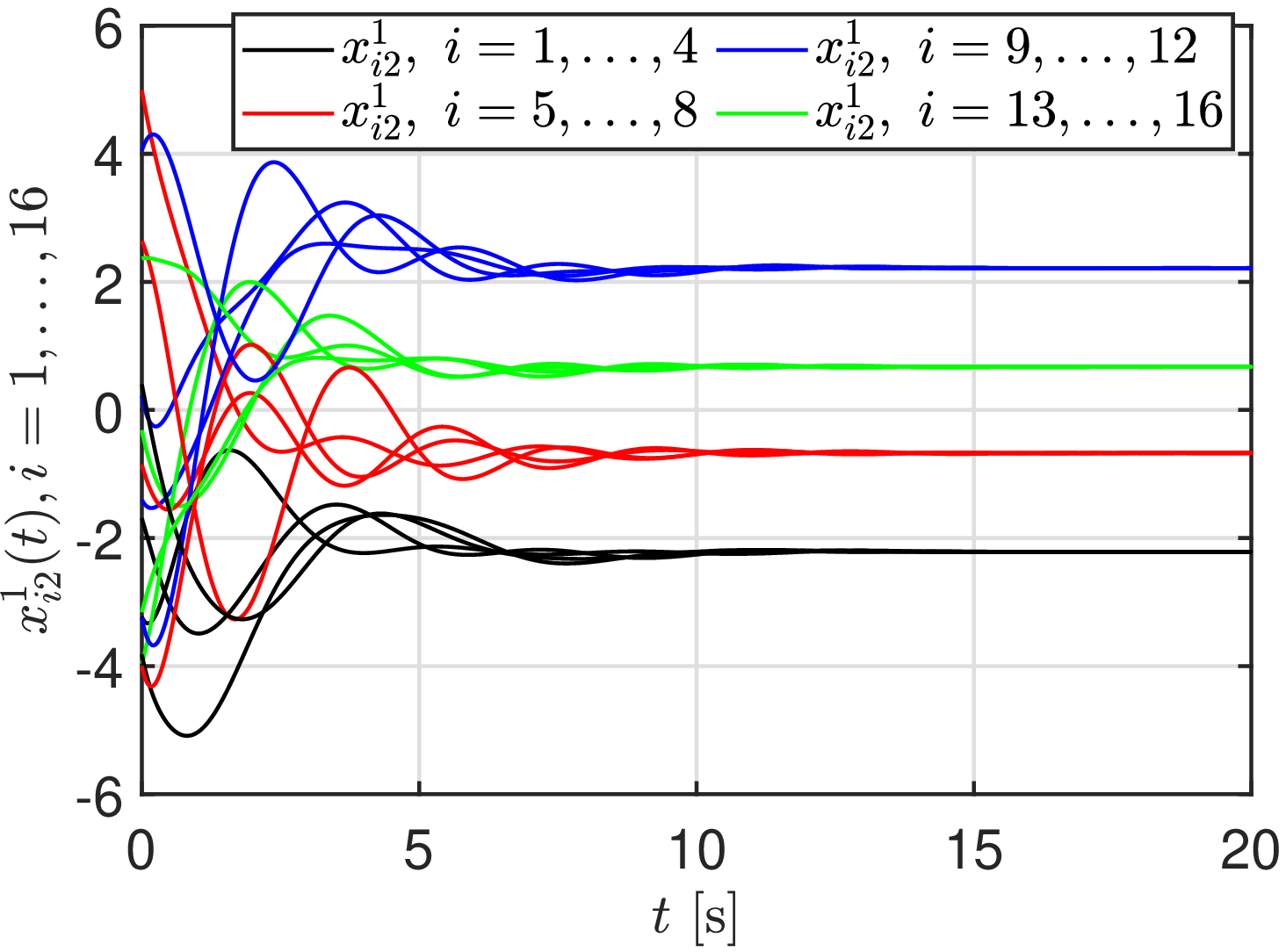}} \hfill
\subfloat[{$x_{i1}^2$ vs $t$ [s]}]{\includegraphics[width=0.24\linewidth]{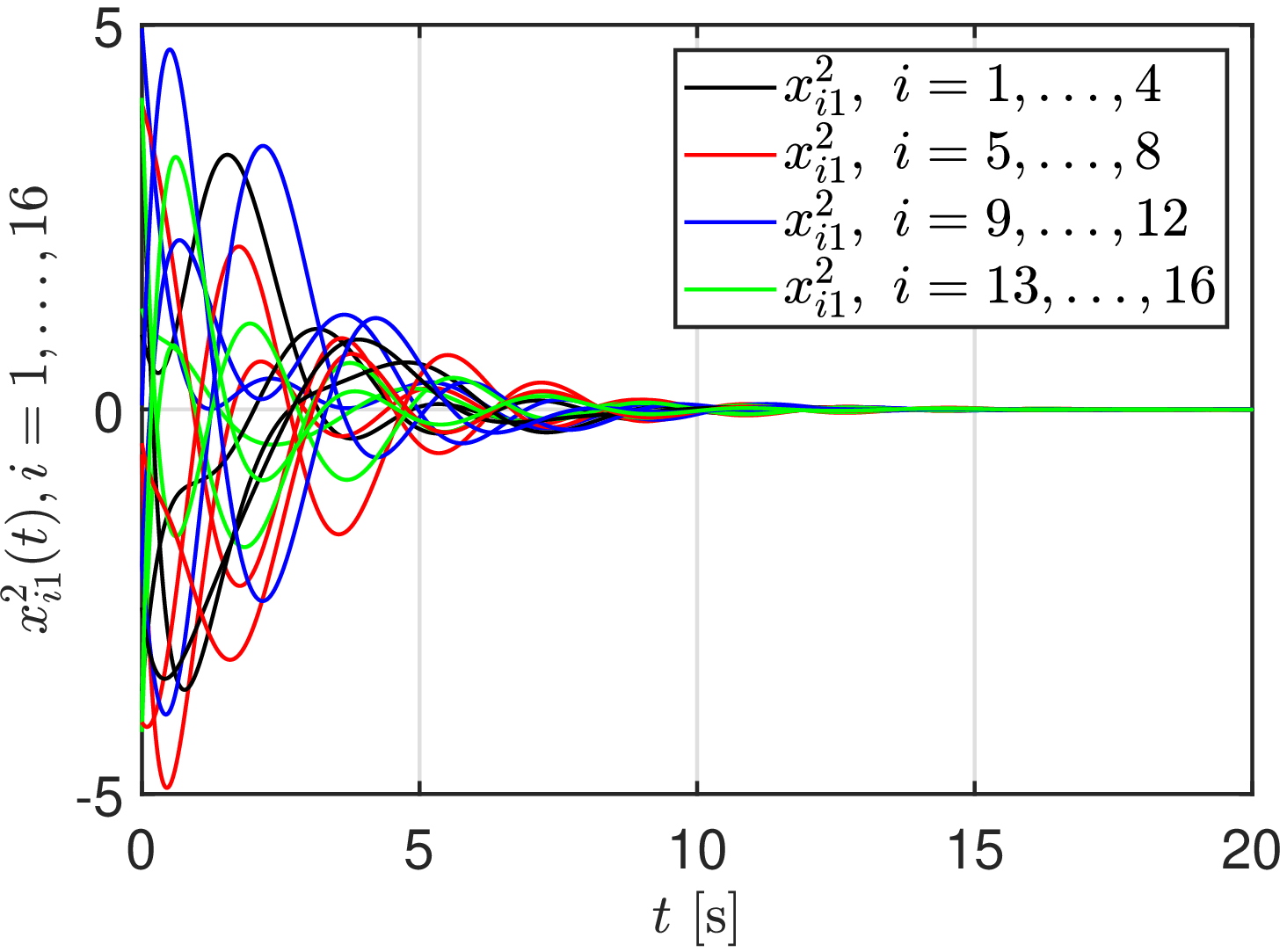}} \hfill
\subfloat[{$x_{i2}^2$ vs $t$ [s]}]{\includegraphics[width=0.24\linewidth]{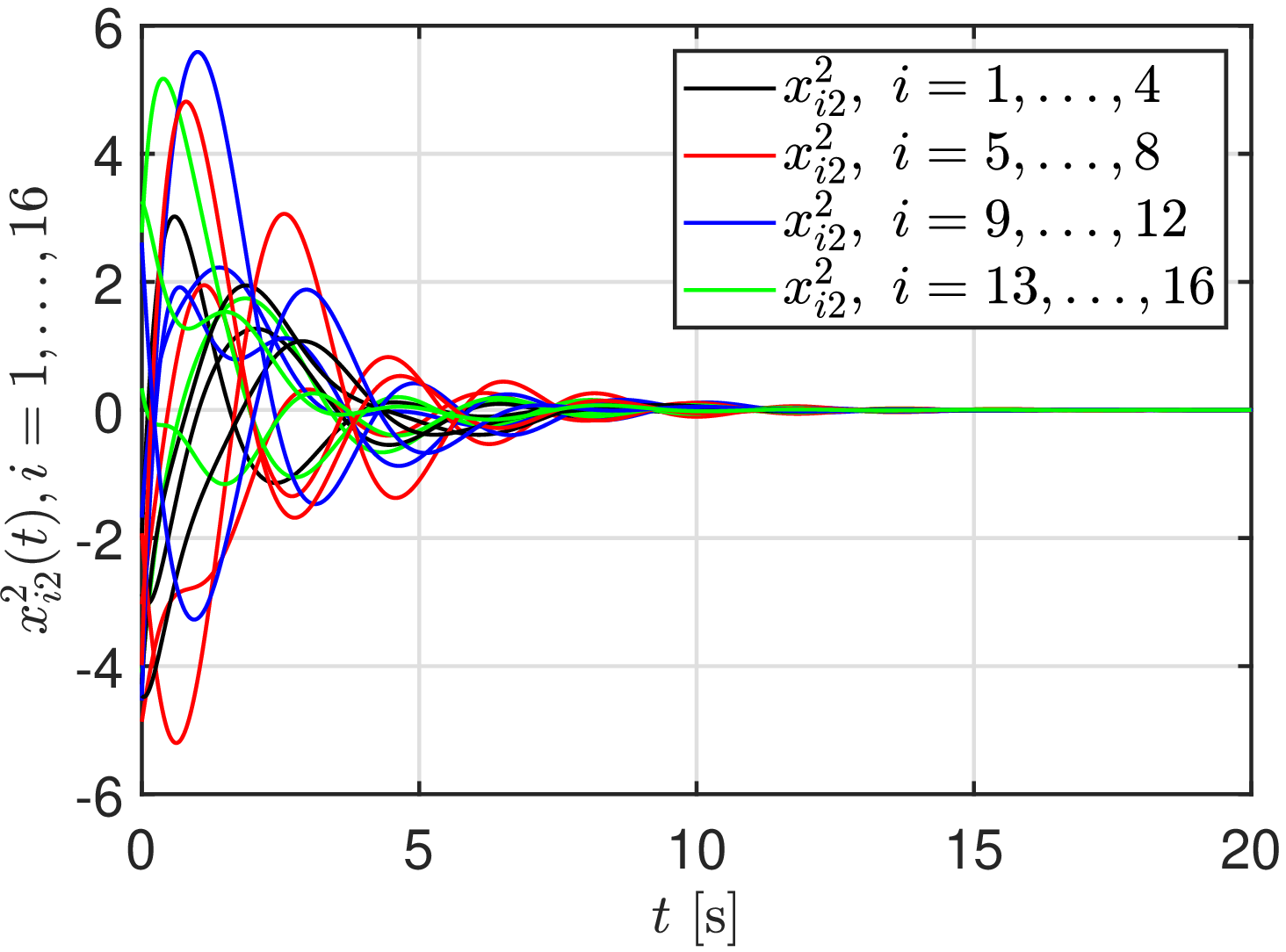}} \hfill
\caption{Simulations of double-integrator agents: (a)--(c): Trajectories of agents with $\alpha = 1.8$,  $1.9724$, and $3$, respectively; (d)--(g): $x_{i1}^k$ and $x_{i2}^k$, $k = 1, 2,$ vs time $t$ [s] corresponding to $\alpha = 1.8$; (h)--(k): $x_{i1}^k$ and $x_{i2}^k$, $k = 1, 2,$ vs time $t$ [s] corresponding to $\alpha = 1.9724$; (l)--(o): $x_{i1}^k$ and $x_{i2}^k$, $k = 1, 2,$ vs time $t$ [s] corresponding to $\alpha = 3$.}
\label{fig:4}
\end{figure*}
\section{Conclusions}
\label{sec:6}
This paper proposed a novel matrix-scaled consensus model, which can describe a multi-dimensional opinion dynamics system, where the heterogeneity in the individuals' private belief systems causes clustering phenomenon frequently. Extension to double-integrator agents were are considered. The matrix scaling gains allow agents keeping their own biased states (in both amplitude and direction) with regard to a virtual consensus point. For further studies, it will be interesting to combine the matrix-scaled consensus with the Altafini model, and study other applications such as scaled synchronization and formation control.

\bibliographystyle{IEEEtran}
\bibliography{minh2021}  

\begin{thebibliography}{10}
\providecommand{\url}[1]{#1}
\csname url@samestyle\endcsname
\providecommand{\newblock}{\relax}
\providecommand{\bibinfo}[2]{#2}
\providecommand{\BIBentrySTDinterwordspacing}{\spaceskip=0pt\relax}
\providecommand{\BIBentryALTinterwordstretchfactor}{4}
\providecommand{\BIBentryALTinterwordspacing}{\spaceskip=\fontdimen2\font plus
\BIBentryALTinterwordstretchfactor\fontdimen3\font minus
  \fontdimen4\font\relax}
\providecommand{\BIBforeignlanguage}[2]{{%
\expandafter\ifx\csname l@#1\endcsname\relax
\typeout{** WARNING: IEEEtran.bst: No hyphenation pattern has been}%
\typeout{** loaded for the language `#1'. Using the pattern for}%
\typeout{** the default language instead.}%
\else
\language=\csname l@#1\endcsname
\fi
#2}}
\providecommand{\BIBdecl}{\relax}
\BIBdecl

\bibitem{Roy2015scaled}
S.~Roy, ``Scaled consensus,'' \emph{Automatica}, vol.~51, pp. 259--262, 2015.

\bibitem{Ren2007magazine}
Ren, Randal, and Atkins, ``Information consensus in multivehicle cooperative
  control,'' \emph{IEEE Control Systems Magazine}, vol.~27, no.~2, pp. 71--82,
  2007.

\bibitem{Freeman2006stability}
Freeman, Peng, and Lynch, ``Stability and convergence properties of dynamic
  average consensus estimators,'' in \emph{45th IEEE Conf Decision Control, CA,
  USA}, 2006, pp. 338--343.

\bibitem{Sarlette2009}
Sarlette and Sepulchre, ``Consensus optimization on manifolds,'' \emph{SIAM
  Journal of Control and Optimization}, vol.~48, no.~1, pp. 56--76, 2009.

\bibitem{Altafini2013}
C.~Altafini, ``Consensus problems on networks with antagonistic interactions,''
  \emph{IEEE Transactions on Automatic Control}, vol.~58, no.~4, pp. 935--946,
  2013.

\bibitem{Trinh2018matrix}
{Trinh et al.}, ``Matrix-weighted consensus and its applications,''
  \emph{Automatica}, vol.~89, pp. 415--419, 2018.

\bibitem{Jadbabaie2003coordination}
Jadbabaie, Lin, and Morse, ``Coordination of groups of mobile autonomous agents
  using nearest neighbor rules,'' \emph{IEEE Transactions on Automatic
  Control}, vol.~48, no.~6, pp. 988--1001, 2003.

\bibitem{Li2009consensus}
{Li et al.}, ``Consensus of multiagent systems and synchronization of complex
  networks: A unified viewpoint,'' \emph{IEEE Transactions on Circuits and
  Systems I: Regular Papers}, vol.~57, no.~1, pp. 213--224, 2009.

\bibitem{Proskurnikov2017tutorial}
Proskurnikov and Tempo, ``A tutorial on modeling and analysis of dynamic social
  networks: {P}art {I},'' \emph{Annual Reviews in Control}, vol.~43, pp.
  65--79, 2017.

\bibitem{Ye2020Aut}
{Ye et al.}, ``Continuous-time opinion dynamics on multiple interdependent
  topics,'' \emph{Automatica}, vol. 115, no. 108884, 2020.

\bibitem{Olfati2007consensuspieee}
Olfati-Saber, Fax, and Murray, ``Consensus and cooperation in networked
  multi-agent systems,'' \emph{Proceedings of the IEEE}, vol.~95, no.~1, pp.
  215--233, 2007.

\bibitem{Meng2015scaled}
Meng and Jia, ``Scaled consensus problems on switching networks,'' \emph{IEEE
  Transactions on Automatic Control}, vol.~61, no.~6, pp. 1664--1669, 2015.

\bibitem{Meng2015TIE}
Meng and Jia, ``Robust consensus algorithms for multiscale coordination control
  of multivehicle systems with disturbances,'' \emph{IEEE Transactions on
  Industrial Electronics}, vol.~63, no.~2, pp. 1107--1119, 2015.

\bibitem{Aghbolagh2016scaled}
Aghbolagh, Ebrahimkhani, and Hashemzadeh, ``Scaled consensus tracking under
  constant time delay,'' \emph{IFAC-PapersOnLine}, vol.~49, no.~22, pp.
  240--243, 2016.

\bibitem{Shang2017delayed}
Y.~Shang, ``On the delayed scaled consensus problems,'' \emph{Applied
  Sciences}, vol.~7, no.~7, p. 713, 2017.

\bibitem{Hanada2019new}
{Kenta et al.}, ``On a new class of structurally balanced graphs for scaled
  group consensus,'' in \emph{58th Annual Conf. Soc. Instrument Control Eng.
  Japan (SICE)}, 2019, pp. 1671--1676.

\bibitem{Wu2021adaptive}
{Wu et al.}, ``Adaptive scaled consensus control of coopetition networks with
  high-order agent dynamics,'' \emph{International Journal of Control},
  vol.~94, no.~4, pp. 909--922, 2021.

\bibitem{Ahn2019consensus}
Ahn and Trinh, ``Consensus under biased alignment,'' \emph{Automatica}, vol.
  110, p. 108605, 2019.

\bibitem{lee2016distributed}
Lee and Ahn, ``Distributed formation control via global orientation
  estimation,'' \emph{Automatica}, vol.~73, pp. 125--129, 2016.

\bibitem{Tran2018ecc}
Tran, Trinh, and Ahn, ``Surrounding formation of star frameworks using
  bearing-only measurements,'' in \emph{Proc. of the European Control
  Conference, Cyprus}, 2018, pp. 368--373.

\bibitem{Godsil2001}
Godsil and Royle, \emph{Algebraic graph theory}.\hskip 1em plus 0.5em minus
  0.4em\relax Springer, 2001.

\bibitem{Slotine1991applied}
Slotine and Li, \emph{Applied Nonlinear Control}.\hskip 1em plus 0.5em minus
  0.4em\relax Prentice Hall, Englewood Cliffs, NJ, 1991, vol. 199, no.~1.

\end{thebibliography}

\end{document}